\title{Real and quaternionic second-order free cumulants and connections to matrix cumulants}
\author{C.\ E.\ I.\ Redelmeier}

\documentclass[11pt]{article}

\usepackage{graphicx}
\usepackage{graphics}
\usepackage{amsfonts}
\usepackage{amssymb}
\usepackage{amsthm}
\usepackage{amsmath}
\usepackage{color}
\usepackage{cite}
\usepackage{enumitem}

\newtheorem{theorem}{Theorem}[section]
\newtheorem{lemma}[theorem]{Lemma}
\newtheorem{proposition}[theorem]{Proposition}

\theoremstyle{remark}

\newtheorem{example}[theorem]{Example}

\theoremstyle{definition}
\newtheorem{definition}[theorem]{Definition}

\begin{document}

\maketitle

\begin{abstract}
We present definitions for real and quaternionic second-order free cumulants, functions whose collective vanishing when applied to elements from different subalgebras is equivalent to the real (resp.\ quaternionic) second-order freeness of those subalgebras.

We show the connection between second-order free cumulants and the topological expansion interpretation of matrix cumulants.  This provides a construction for higher-order free cumulants.  Coefficients are given in terms of the asymptotics of the cumulants of the Weingarten function.
\end{abstract}

\section{Introduction}

The connection between random matrices and free probability was first presented in \cite{MR1094052}.  This type of construction was extended to the fluctuations and higher-order statistics of random matrices in \cite{MR2216446, MR2294222, MR2302524} in the complex case and in \cite{MR3217665, 2015arXiv150404612R} in the real and quaternionic cases.  (Unlike in the first-order case, where the asymptotic behaviour of random matrices is identical in all three cases, the asymptotic higher-order behaviour depends on whether the matrices are real, complex, or quaternionic, or more accurately on whether the probability distribution of the random matrix is orthogonally invariant, unitarily invariant, or symplectically invariant, respectively.)

The free cumulants of Speicher (see, e.g.\ \cite{MR2266879}) are quantities which, when applied to mutually free elements will vanish.  They can thus be used to test for freeness.  They can also be used to compute moments of elements generated by free subalgebras when the moments are known on each subalgebra.

The matrix cumulants of \cite{MR2240781, MR2337139, MR2483727} provide a similar construction for finite-dimensional random matrices.  These are constructed in terms of a convolution as well as in terms of a geometric projection.  These quantities are multiplicative over algebras of random matrices which are independent and unitarily (respectively orthogonally, symplectically) invariant (this definition also depends on the symmetry of the random matrix's probability distribution), and can thus be used to compute moments of expressions in the algebra generated by ensembles of independent matrices in general position if the moments, and hence the matrix cumulants, of the individual ensembles are known.  In these papers, it is shown that the asymptotic values of the matrix cumulants are the free cumulants of the limit distribution of the random matrices.

\begin{figure}
\centering
\begin{tabular}{c}
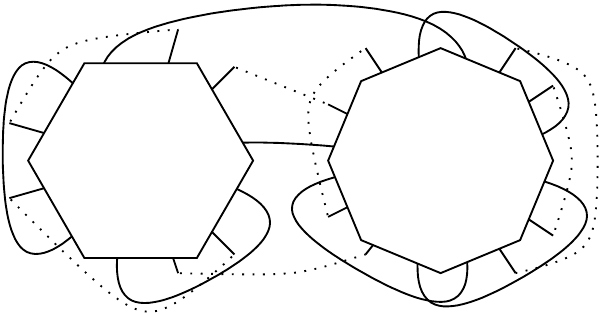\\
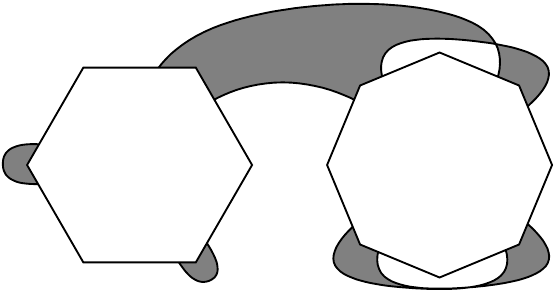
\end{tabular}
\caption{A summand in the topological expansion of an expression with Haar-distributed orthogonal matrces \(O\) and arbitrary orthogonally invariant random matrices \(X_{1},\ldots,X_{7}\) (top), and the corresponding surface gluing in the topological expansion for the \(X_{1},\ldots,X_{7}\) (bottom).}
\label{figure: orthogonally invariant}
\end{figure}

\begin{figure}
\centering
\begin{tabular}{cc}
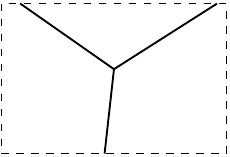&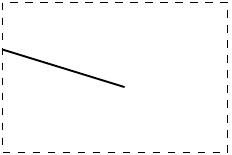\\
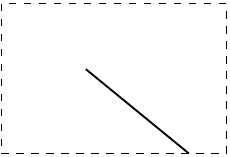&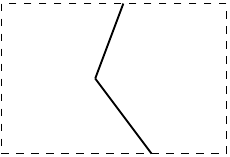
\end{tabular}
\caption{The vertices which appear on the surface constructed in Figure~\ref{figure: orthogonally invariant}, bottom.}
\label{figure: vertices}
\end{figure}

The connection between matrix cumulants and topological expansion is described in \cite{MR3217665, 2015arXiv151101087R}.  If \(X_{1},\ldots,X_{7}\) are random matrices with orthogonally invariant distributions, a matrix integral which is a product of traces, such as, for example,
\begin{equation}
\mathbb{E}\left(\mathrm{tr}\left(X_{1}X_{2}X_{3}\right)\mathrm{tr}\left(X_{4}X_{5}X_{6}X_{7}\right)\right)
\label{formula: example moment}
\end{equation}
may be computed by summing over surfaces glued from faces corresponding to the traces (see Figure~\ref{figure: orthogonally invariant}).  The lower diagram shows such a gluing.  The summand corresponding to this surface is
\[N^{2F-\chi}c_{\pi}\left(X_{1},\ldots,X_{7}\right)\]
where \(N\) is the dimension of the matrix, \(F\) is the number of traces in the integral, \(\chi\) is the Euler characteristic of the surface, and \(c_{\pi}\left(X_{1},\ldots,X_{7}\right)\) is the cumulant associated to the collection of four vertices which appear on the resulting surface (see Figure~\ref{figure: vertices}; note \(X_{4}^{T}\) and \(X_{7}^{T}\) which appear on the ``backs'' of the corners containing \(X_{4}\) and \(X_{7}\), respectively).

The topological expansion for orthogonally invariant matrices follows from the topological expansion for Haar-distributed orthogonal matrices \(O\), which may be expressed as a sum over pairings on each of the first and second indices of \(O\) \cite{MR2217291,MR2337139,2015arXiv151101087R}.  A matrix cumulant is the sum over pairings on the second index for a fixed choice of pairing on the first index.  (An example is shown in the top part of Figure~\ref{figure: orthogonally invariant}, with the pairing on the first indices shown in solid lines, and a pairing on the second indices marked with dotted lines.)  This cumulant corresponds to the gluing shown in the lower part of Figure~\ref{figure: orthogonally invariant}.  See Example~\ref{example: gluing} for more detail.

If any of the \(X_{k}\) are independent and orthogonally in general position, they may be conjugated by independent Haar-distributed matrices, say \(O_{1}\) and \(O_{2}\).  In this case, we do not need to consider terms in which independent Haar matrices are connected (by either index), so any term in which \(O_{1}\) and \(O_{2}\) are connected will vanish, and the contribution of the diagram will be the product of the contribution of the \(O_{1}\) vertices and the contribution of the \(O_{2}\) vertices.  Thus, independent matrices in general position can be thought of as different colours, and only gluings which respect this colouring need be considered.  For example, in Figure~\ref{figure: orthogonally invariant}, if \(X_{1}\) and \(X_{4}\) are independent and orthogonally in general position, the term shown will vanish.  The contribution of each colour is multiplicative.  If \(X_{1}\), \(X_{2}\), \(X_{4}\), and \(X_{5}\) as a set are independent from \(X_{3}\), \(X_{6}\) and \(X_{7}\), then the contribution of the lower diagram in Figure~\ref{figure: orthogonally invariant} will be the product of the contribution of the two upper vertices in Figure~\ref{figure: vertices} multiplied by the contribution of the two lower vertices.  In this way, if the matrix cumulants of independent ensembles in general position are known, it is possible to compute the matrix cumulants, and hence the moments such as (\ref{formula: example moment}), of any expression generated by those ensembles.  

The contribution is not in general, however, multiplicative over the vertices (i.e.\ the matrix cumulant is not the product of the four matrix cumulants associated to the four vertices).  In this topological interpretation, it is natural to ask what the difference is between the product of the cumulants of the vertices individually or in subsets and the cumulant of the vertices together.  A natural way to do so is similar to the construction of classical cumulants, where we consider the quantities such as the two-vertex cumulant
\[c_{V_{1},V_{2}}\left(X_{1},\ldots,X_{n}\right)-c_{V_{1}}\left(X_{1},\cdots,X_{n}\right)c_{V_{2}}\left(X_{1},\ldots,X_{n}\right)\]
where \(c_{V_{1},V_{2}}\) is the cumulant corresponding to the two vertices \(V_{1}\) and \(V_{2}\), and \(c_{V_{i}}\) is the cumulant corresponding to vertex \(V_{i}\) (see Definition~\ref{definition: vertex cumulants} for the general construction).  These quantities can be thought of as the difference of the cumulants from being multiplicative.

While these quantities do not in general vanish, if the matrices satisfy cerain convergence properties, the cumulants on a larger number of vertices are lower-order: specifically, the vertex cumulants on one vertex are \({\cal O}\left(1\right)\), while the vertex cumulants on \(V\) vertices are \({\cal O}\left(N^{2-2V}\right)\), where \(N\) is the dimension of the matrix.  In this sense, the matrix cumulants are asymptotically multiplicative.  In addition, these cumulants have the desirable property of vanishing when they are applied to independent matrices in general position.  (Since a matrix cumulant on vertices of different colours is the product of each colour, this difference from multiplicativity will vanish.)

Finally, we find that the asymptotic limit of the appropriately renormalized two-vertex cumulants are the second-order real free cumulants, quantities which vanish when applied to elements which are second-order free.

In Section~\ref{section: preliminaries}, we define notation and introduce relevant lemmas.  (More specialized notation and lemmas, in particular those pertaining to maps on nonorientable surfaces and their applications in topological expansions, and those pertaining to higher-order freeness, are gathered at the beginning of Section~\ref{section: matrix}.)  In Section~\ref{section: freeness}, we present expressions for the second-order real cumulants, and demonstrate that the vanishing of mixed cumulants is equivalent to second-order real freeness.  In Section~\ref{section: matrix} we construct the vertex cumulants from the matrix cumulants and present expressions and properties.  In Subsection~\ref{subsection: asymptotics} we discuss the asymptotic properties of the vertex cumulants.  In Subsection~\ref{subsection: second-order} we discuss the second-order case and show that the asymptotic limit of a two-vertex cumulant is a second-order cumulant.  Generalization to higher-order freeness is discussed in \ref{subsection: higher-order}.  Section~\ref{section: freeness} and Section~\ref{section: matrix} may be read somewhat independently from each other.  In Section~\ref{section: quaternion}, we outline the construction in the quaternionic case.

\section{Preliminaries}

\label{section: preliminaries}

\subsection{Partitions and permutations}

We denote the set of integers \(\left\{1,\ldots,n\right\}\) by \(\left[n\right]\) and the set of integers \(\left\{m,\ldots,n\right\}\) by \(\left[m,n\right]\).

For a finite set \(I\), we denote the number of elements in \(I\) by \(\left|I\right|\).

\begin{definition}
A {\em partition} of a set \(I\) is a set \(\left\{U_{1},\ldots,U_{k}\right\}\) of nonempty subsets of \(I\) called {\em blocks} such that \(U_{i}\cap U_{j}=\emptyset\) for \(i\neq j\) and \(U_{1}\cup\cdots\cup U_{k}=I\).  We denote the number of blocks of a partition \({\cal U}\) by \(\#\left({\cal U}\right)\).  The set of partitions of \(I\) is denoted \({\cal P}\left(I\right)\) and the set of partitions of \(\left[n\right]\) by \({\cal P}\left(n\right)\).  The partitions may be given a partial order by letting \({\cal U}\preceq{\cal V}\) when every block of \({\cal U}\) is a subset of a block of \({\cal V}\).

We define the restriction of a partition \({\cal U}=\left\{U_{1},\ldots,U_{k}\right\}\in{\cal P}\left(I\right)\) to a subset \(J\subseteq I\) by letting \(\left.{\cal U}\right|_{J}:=\left\{U_{1}\cap J,\ldots,U_{k}\cap J\right\}\) (excluding empty sets).

We denote the set of {\em pairings}, or partitions where each block contains \(2\) elements, of set \(I\) by \({\cal P}_{2}\left(I\right)\), and denote \({\cal P}_{2}\left(\left[n\right]\right)\) by \({\cal P}_{2}\left(n\right)\).
\end{definition}

\begin{definition}
We denote the set of permutations on set \(I\) by \(S\left(I\right)\), and \(S\left(\left[n\right]\right)\) by \(S\left(n\right)\).

The orbits (cycles) of a permutation form a partition of its domain.  For \(\pi\in S\left(I\right)\) we denote this partition by \(\Pi\left(\pi\right)\).  When it is clear, we will sometimes use \(\pi\) to denote this partition as well as the permutation.

We note that \(\pi^{-1}\), as well as any conjugate of \(\pi\), have the same cycle structure as \(\pi\).

For \(\pi\in S\left(I\right)\) and subset \(J\subseteq I\), the permutation induced by \(\pi\) on \(J\), denoted \(\left.\pi\right|_{J}\), is the permutation where \(\left.\pi\right|_{J}\left(a\right)=\pi^{k}\left(a\right)\), where \(k>1\) is the smallest exponent such that \(\pi^{k}\left(a\right)\in J\).  (This is the permutation obtained from \(\pi\) in cycle notation by deleting all elements not in \(J\)).

\subsection{Integer sequences}

\begin{definition}
The Catalan numbers are given by
\[C_{n}:=\frac{1}{n+1}\binom{2n}{n}\textrm{.}\]
\end{definition}

The following integers appear in \cite{MR1761777, MR1959915, MR2217291}:
\begin{definition}
Let \({\cal U},{\cal V}\in{\cal P}\left(I\right)\) with \({\cal U}\preceq{\cal V}\).  Define
\begin{multline*}
\gamma_{{\cal U},{\cal V}}:=\prod_{V\in{\cal V}}\left(-1\right)^{\left|V\right|-\#\left(\left.{\cal U}\right|_{V}\right)}\frac{2^{2\#\left(\left.{\cal U}\right|_{V}\right)-1}\left(2\left|V\right|+\#\left(\left.{\cal U}\right|_{V}\right)-3\right)!}{\left(2\left|V\right|\right)!}\\\times\prod_{U\in\left.{\cal U}\right|_{V}}\frac{\left(2\left|U\right|-1\right)!}{\left(\left|U\right|-1\right)!^{2}}\textrm{.}
\end{multline*}
\end{definition}

We note that the contribution of a block \(V\in{\cal V}\) containing only one block of \({\cal U}\) is \(\left(-1\right)^{\left|V\right|-1}C_{\left|V\right|-1}\), and the contribution of a block \(V\in{\cal V}\) containing two blocks \(U_{1},U_{2}\in{\cal U}\) is
\[\left(-1\right)^{\left|U_{1}\right|+\left|U_{2}\right|}\frac{4}{\left|U_{1}\right|+\left|U_{2}\right|}\frac{\left(2\left|U_{1}\right|-1\right)!}{\left(\left|U_{1}\right|-1\right)!^{2}}\frac{\left(2\left|U_{2}\right|-1\right)!}{\left(\left|U_{2}\right|-1\right)!^{2}}\textrm{.}\]

\subsection{Cartographic machinery}

For integers \(r_{1},\ldots,r_{k}\), we define the permutation in \(S\left(r_{1}+\cdots+r_{k}\right)\)
\[\tau_{r_{1},\ldots,r_{k}}:=\left(1,\ldots,r_{1}\right)\cdots\left(r_{1}+\cdots+r_{k-1}+1,\ldots,r_{1}+\cdots+r_{k}\right)\textrm{.}\]

\begin{definition}
Given a permutation \(\pi\), the {\em Kreweras complement} of permutation \(\rho\) with respect to \(\pi\) is
\[\mathrm{Kr}_{\pi}\left(\rho\right):=\rho^{-1}\pi\textrm{.}\]
We write \(\mathrm{Kr}_{p,q}\left(\pi\right)\) for \(\mathrm{Kr}_{\tau_{p,q}}\left(\pi\right)\), or omit the subscript if it is clear from context.
\end{definition}
(For more on the topological and graph theoretical interpretation of this and the related constructions, see e.g.\ \cite{MR1813436, MR0404045, MR2036721, 2012arXiv1204.6211R}.)

The Euler characteristic of \(\pi\) (on \(\rho\)) is
\[\chi_{\rho}\left(\pi\right):=\#\left(\rho\right)+\#\left(\pi\right)+\#\left(\mathrm{Kr}_{\rho}\left(\pi\right)\right)-n\textrm{.}\]
The Euler characteristic is less than or equal to \(2\) times the number of connected components:
\[\chi_{\rho}\left(\pi\right)\leq 2\#\left(\Pi\left(\pi\right)\vee\Pi\left(\rho\right)\right)\]
(see \cite{MR2052516}).
A permutation is said to be {\em noncrossing} (on \(\rho\)) if
\[\chi_{\rho}\left(\pi\right)=2\#\left(\Pi\left(\pi\right)\vee\Pi\left(\rho\right)\right)\textrm{.}\]
We denote the set of \(\pi\) noncrossing on \(\rho\) by \(S_{\mathrm{nc}}\left(\rho\right)\), and \(S_{\mathrm{nc}}\left(p,q\right):=S_{\mathrm{nc}}\left(\tau_{p,q}\right)\).  We denote the set of \(\pi\) such that \(\Pi\left(\pi\right)\preceq\Pi\left(\rho\right)\) by \(S_{\mathrm{disc-nc}}\left(\rho\right)\) (and let \(S_{\mathrm{disc-nc}}\left(p,q\right):=S_{\mathrm{disc-nc}}\left(\tau_{p,q}\right)\)).  We denote the \(\pi\in S_{\mathrm{nc}}\left(p,q\right)\) such that at least one cycle of \(\pi\) contains elements from both \(\left[p\right]\) and \(\left[p+1,p+q\right]\) by \(S_{\mathrm{ann-nc}}\left(p,q\right)\).

A poset may be constructed on \(S_{\mathrm{nc}}\left(\tau\right)\) by letting \(\pi\preceq\rho\) if \(\Pi\left(\pi\right)\preceq\Pi\left(\rho\right)\) and \(\pi\in S_{\mathrm{nc}}\left(\rho\right)\).  If \(\tau\) has only one cycle, then this poset is the same as the poset on the \(\Pi\left(\rho\right)\).  This poset is a lattice (any two elements \(\pi\) and \(\rho\) have a unique \(\sup\) \(\pi\vee\rho\) and \(\inf\) \(\pi\wedge\rho\)).  It is also self-dual.  (See \cite{MR2266879}, Chapters~9 and 10.)
\end{definition}

In the cases where \(\tau\) has \(1\) or \(2\) cycles, \(\pi\) being noncrossing is equivalent to a number of conditions on a fixed number of elements not occurring.

\begin{lemma}[Biane]
Let \(\tau\in S\left(I\right)\) have one cycle.  Then \(\pi\in S_{\mathrm{nc}}\left(\tau\right)\) iff neither of the following conditions occurs:
\begin{enumerate}
	\item There are \(a,b,c\in I\) such that \(\left.\tau\right|_{\left\{a,b,c\right\}}=\left(a,b,c\right)\) and \(\left.\pi\right|_{\left\{a,b,c\right\}}=\left(a,c,b\right)\).
	\item There are \(a,b,c,d\in I\) such that \(\left.\tau\right|_{\left\{a,b,c,d\right\}}=\left(a,b,c,d\right)\) and \(\left.\pi\right|_{\left\{a,b,c,d\right\}}=\left(a,c\right)\left(b,d\right)\).
\end{enumerate}
\label{lemma: disc noncrossing}
\end{lemma}
See \cite{MR1475837}.

\begin{lemma}[Mingo, Nica]
Let \(\tau:=\tau_{p,q}\).

The annular nonstandard conditions are:
\begin{enumerate}
	\item There are \(a,b,c\in\left[p+q\right]\) with \(\left.\tau\right|_{\left\{a,b,c\right\}}\left(a,b,c\right)\) and \(\left.\pi\right|_{\left\{a,b,c\right\}}\left(a,c,b\right)\).
\label{item: ans1}
	\item There are \(a,b,c,d\in\left[p+q\right]\) with \(\left.\tau\right|_{\left\{a,b,c,d\right\}}\left(a,b\right)\left(c,d\right)\) and \(\left.\pi\right|_{\left\{a,b,c,d\right\}}\left(a,c,b,d\right)\).
\label{item: ans2}
\end{enumerate}

For \(x,y\) in different cycles of \(\tau\), we let
\[\lambda_{x,y}:=\left(\tau\left(x\right),\tau^{2}\left(x\right),\ldots,\tau^{-1}\left(x\right),\tau\left(y\right),\tau^{2}\left(y\right),\ldots,\tau^{-1}\left(y\right)\right)\textrm{.}\]
The annular noncrossing conditions are:
\begin{enumerate}
	\item There are \(a,b,c,d\in\left[p+q\right]\) with \(\left.\tau\right|_{\left\{a,b,c,d\right\}}\left(a,b,c,d\right)\) and \(\left.\pi\right|_{\left\{a,b,c,d\right\}}=\left(a,c\right)\left(b,d\right)\).
\label{item: ac1}
	\item There are \(a,b,c,x,y\in\left[p+q\right]\) such that \(\left.\lambda_{x,y}\right|_{\left\{a,b,c\right\}}=\left(a,b,c\right)\) and \(\left.\pi\right|_{\left\{a,b,c,x,y\right\}}=\left(a,c,b\right)\left(x,y\right)\).
\label{item: ac2}
	\item There are \(a,b,c,d,x,y\in\left[p+q\right]\) such that \(\left.\lambda_{x,y}\right|_{\left\{a,b,c,d\right\}}=\left(a,b,c,d\right)\) and \(\left.\pi\right|_{\left\{a,b,c,d,x,y\right\}}=\left(a,c\right)\left(b,d\right)\left(x,y\right)\).
\label{item: ac3}
\end{enumerate}

Then \(\pi\in S_{\mathrm{ann-nc}}\left(p,q\right)\) iff none of the annular nonstandard or annular crossing conditions occurs.
\end{lemma}
See \cite{MR2052516}.

\begin{definition}
Let \(p\) and \(q\) be fixed.  For a permutation \(\pi\in S\left(p+q\right)\), we define \(\pi_{\mathrm{op}_{1}}\) to be the permutation constructed from \(\pi\) by switching \(1\) with \(p\), \(2\) with \(p-1\), and so on, in the cycle notation.  Likewire, we define \(\pi_{\mathrm{op}_{2}}\) to be the permutation constructed from \(\pi\) by switching \(p+1\) with \(p+q\), \(p+2\) with \(p+q-1\), and so on.
\label{definition: opposite}
\end{definition}
If \(\pi\in S_{\mathrm{nc}}\left(p,q\right)\), then for \(\tau:=\tau_{p,q}\), \(\pi_{\mathrm{op}_{i}}\in S_{\mathrm{nc}}\left(\tau_{\mathrm{op}_{i}}\right)\).  We note that \(\mathrm{Kr}_{\tau_{\mathrm{op}_{i}}}\left(\pi_{\mathrm{op}_{i}}\right)\) has the same cycle structure as \(\mathrm{Kr}_{p,q}\left(\pi\right)\).

\begin{definition}
We denote the set of \(\left({\cal U},\pi\right)\in{\cal P}\left(p+q\right)\times S_{p+q}\) such that \({\cal U}\succeq\Pi\left(\pi\right)\) by \(\mathit{PS}\left(p,q\right)\).  We denote by \(\mathit{PS}^{\prime}\left(p,q\right)\) the set of \(\left({\cal U},\pi\right)\in\mathit{PS}\left(p,q\right)\) such that \(\pi\in S_{\mathrm{nc}}\left(p,q\right)\), and either \({\cal U}=\Pi\left(\pi\right)\), or each block of \({\cal U}\) contains exactly one block of \(\Pi\left(\pi\right)\), except one block \(U\in{\cal U}\) which contains exactly two blocks of \(\Pi\left(\pi\right)\), one contained in \(\left[p\right]\) and the other contained in \(\left[p+1,p+q\right]\).
\end{definition}

The following construction is described in more detail in \cite{2020arXiv200309344R}, including computations of the M\"{o}bius function.  (For definitions and standard results on M\"{o}bius functions see a combinatorics text such as \cite{MR1311922}.)
\begin{definition}
We define \(S_{\mathrm{sd}}\left(p,q\right)\) as the disjoint union of \(S_{\mathrm{ann-nc}}\left(p,q\right)\) with two copies of \(S_{\mathrm{disc-nc}}\left(p,q\right)\), denoted \(S_{\mathrm{nc}}\left(p,q\right)\) and \(\hat{S}_{\mathrm{nc}}\left(p,q\right)\).  If \(\pi\in S_{\mathrm{disc-nc}}\left(p,q\right)\), we denote the same permutation in \(\hat{S}_{\mathrm{disc-nc}}\left(p,q\right)\) by \(\hat{\pi}\) (however we may nonetheless use a symbol like \(\pi\) for an element in \(\hat{S}_{\mathrm{disc-nc}}\left(p,q\right)\), and permutation operations such as \(\mathrm{Kr}\left(\pi\right)\) refer to the underlying permutation).

We define a partial order on \(S_{\mathrm{sd}}\left(p,q\right)\): if \(\pi,\rho\in S_{\mathrm{disc-nc}}\left(p,q\right)\cup S_{\mathrm{ann-nc}}\left(p,q\right)\), then \(\pi\preceq\rho\) if \(\pi\preceq\rho\) in the usual partial order on permutations, and if \(\pi\in S_{\mathrm{sd}}\left(p,q\right)\) and \(\rho\in\hat{S}_{\mathrm{disc-nc}}\left(p,q\right)\), then \(\pi\preceq\rho\) if \(\mathrm{Kr}\left(\rho\right)\succeq\mathrm{Kr}\left(\pi\right)\) in the usual partial order.

We define an order reversing bijection \(\widehat{\mathrm{Kr}}:S_{\mathrm{sd}}\left(p,q\right)\rightarrow S_{\mathrm{sd}}\left(p,q\right)\) (resp.\ \(\widehat{\mathrm{Kr}}^{-1}\)) taking \(S_{\mathrm{disc-nc}}\left(p,q\right)\) to \(\hat{S}_{\mathrm{disc-nc}}\left(p,q\right)\) and {\em vice versa}, where the underlying permutation of \(\widehat{\mathrm{Kr}}\left(\pi\right)\) is the underlying permutation of \(\mathrm{Kr}\left(\pi\right)\).

We denote the M\"{o}bius function on \(S_{\mathrm{sd}}\left(p,q\right)\) by \(\mu_{\mathrm{sd}}\).  For \(\pi,\rho\in S_{\mathrm{sd}}\left(p,q\right)\) with \(\pi\preceq\rho\), if it is not the case that \(\pi\in S_{\mathrm{disc-nc}}\left(p,q\right)\) and \(\rho\in \hat{S}_{\mathrm{disc-nc}}\left(p,q\right)\), then
\[\mu_{\mathrm{sd}}\left(\pi,\rho\right)=\prod_{U\in\Pi\left(\mathrm{Kr}_{\rho}\left(\pi\right)\right)}\left(-1\right)^{\left|U\right|-1}C_{\left|U\right|-1}\textrm{;}\]
and if \(\pi\in S_{\mathrm{disc-nc}}\left(p,q\right)\) and \(\rho\in\hat{S}_{\mathrm{disc-nc}}\left(p,q\right)\), then
\begin{multline*}
\mu_{\mathrm{sd}}\left(\pi,\rho\right)\\=\sum_{\substack{U_{1},U_{2}\in\Pi\left(\mathrm{Kr}_{\rho}\left(\pi\right)\right)\\U_{1}\in\left[p\right],U_{2}\in\left[p+1,p+q\right]}}\left(-1\right)^{\left|U_{1}\right|+\left|U_{2}\right|}\frac{2}{\left|U_{1}\right|+\left|U_{2}\right|}\frac{\left(2\left|U_{1}\right|-1\right)!}{\left(\left|U_{1}\right|-1\right)!^{2}}\frac{\left(2\left|U_{2}\right|-1\right)!}{\left(\left|U_{2}\right|-1\right)!^{2}}\\\times\prod_{U\in\Pi\left(\mathrm{Kr}_{\rho}\left(\pi\right)\right)\setminus\left\{U_{1},U_{2}\right\}}\left(-1\right)^{\left|U\right|-1}C_{\left|U\right|-1}-\prod_{U\in\Pi\left(\mathrm{Kr}_{\rho}\left(\pi\right)\right)}\left(-1\right)^{\left|U\right|-1}C_{\left|U\right|-1}\textrm{.}
\end{multline*}
\end{definition}
We note that, for any \(\pi,\rho\in S_{\mathrm{disc-nc}}\left(p,q\right)\) with \(\pi\preceq\rho\)
\[\mu_{\mathrm{sd}}\left(\pi,\rho\right)+\mu_{\mathrm{sd}}\left(\pi,\hat{\rho}\right)=\frac{1}{2}\sum_{{\cal U}:\left({\cal U},\mathrm{Kr}_{\rho}\left(\pi\right)\right)\in{\cal PS}^{\prime}\left(p,q\right)}\gamma_{\mathrm{Kr}_{\rho}\left(\pi\right),{\cal U}}\textrm{.}\]

We also note that \(\mu_{\mathrm{sd}}\left(\pi,\rho\right)=\mu_{\mathrm{sd}}\left(\pi_{\mathrm{op}_{i}},\rho\right)\) for any \(\pi,\rho\in S_{\mathrm{sd-nc}}\left(p,q\right)\).

\subsection{Freeness}

\begin{definition}
A {\em noncommutative probability space} is a unital algebra \(A\) equipped with a linear functional \(\phi_{1}:A\rightarrow\mathbb{F}\) such that \(\phi_{1}\left(1\right)=1\).

An {\em \(n\)th-order probability space} has, in addition, functionals \(\varphi_{1},\ldots,\varphi_{n}\) such that \(\varphi_{i}:A^{i}\rightarrow\mathbb{F}\) is linear in each argument and vanishes whenever any argument is \(1\).

A {\em real second-order probability space} must also have an involution \(x\mapsto x^{t}\) such that for any \(x,y\in A\), \(\left(xy\right)^{t}=y^{t}x^{t}\), and the value of a \(\varphi_{k}\left(x_{1},\ldots,x_{k}\right)\) is unchanged if any of the \(x_{i}\) is replaced with \(x_{i}^{t}\).
\end{definition}

\begin{definition}
We say that an element of a noncommutative probability space \(x\in A\) is {\em centred} if \(\varphi_{1}\left(x\right)=0\).  We denote the centred random variable \(\mathring{x}:=x-\varphi_{1}\left(x\right)1\).

If noncommutative probability space has subalgebras \(A_{1},\ldots,A_{n}\), we say that variables \(x_{1},\ldots,x_{p}\in A_{1}\cup\cdots\cup A_{n}\) are {\em alternating} if \(x_{i}\) comes from a different algebra from \(x_{i+1}\), \(i=1,\ldots,p-1\).  If in addition \(x_{p}\) comes from a different algebra from \(x_{1}\), we say that they are {\em cyclically alternating}.
\end{definition}

\begin{definition}
\label{definition: second-order freeness}
Subalgebras \(A_{1},\ldots,A_{n}\subseteq A\) of a noncommutative probability space are free if, for any \(x_{1},\ldots,x_{p}\) centred and alternating,
\[\varphi_{1}\left(x_{1}\cdots x_{p}\right)=0\textrm{.}\]

Subalgebras of a second-order probability space are real second-order free if they are free, and if, for \(x_{1},\ldots,x_{p}\) and \(y_{1},\ldots,y_{q}\) centred and cyclically alternating,
\begin{multline}
\varphi_{2}\left(x_{1},\ldots,x_{p},y_{1},\ldots,y_{q}\right)\\=\left\{\begin{array}{ll}\sum_{k=1}^{p}\prod_{i=1}^{p}\varphi_{1}\left(x_{i}y_{\tau_{p,q}^{-k}\left(i\right)}\right)+\sum_{k=1}^{p}\prod_{i=1}^{p}\varphi_{1}\left(x_{i}y_{\tau_{p,q}^{k}\left(i\right)}\right)&p=q\\0\textrm{,}&p\neq q\end{array}\right.
\label{formula: second-order freeness}
\end{multline}
(see \cite{MR3217665} for diagrams representing this formula).
\end{definition}

\begin{definition}
Let \(\left(A,\varphi\right)\) be a noncommutative probability space, and let \(x_{1},\ldots,x_{n}\in A\).  We let
\[\alpha_{n}\left(x_{1},\ldots,x_{n}\right):=\varphi_{1}\left(x_{1}\cdots x_{n}\right)\textrm{.}\]

Let \(I\subseteq\left[n\right]\), and let
\[\pi=\left(c_{1},\ldots,c_{k_{1}}\right)\cdots\left(c_{k_{1}+\cdots+k_{m-1}+1},\ldots,c_{k_{1}+\cdots+k_{m}}\right)\in{\cal P}\left(I\right)\textrm{.}\]
We let
\[\alpha_{\pi}\left(x_{1},\ldots,x_{n}\right):=\varphi_{1}\left(x_{c_{1}}\cdots x_{c_{k_{1}}}\right)\cdots\varphi_{1}\left(x_{c_{k_{1}+\cdots+k_{m-1}+1}}\cdots x_{c_{k_{1}+\cdots+k_{m}}}\right)\textrm{.}\]
\end{definition}

\begin{definition}[Speicher]
For \(n=1,2,\ldots,\) and for \(I\subseteq\left[n\right]\) and
\[\pi=\left(c_{1},\ldots,c_{k_{1}}\right)\cdots\left(c_{k_{1}+\cdots+k_{m-1}+1},\ldots,c_{k_{1}+\cdots+k_{m}}\right)\in{\cal P}\left(I\right)\textrm{.}\]
define the {\em free cumulants} \(\kappa_{n}:A^{n}\rightarrow\mathbb{C}\) and \(\kappa_{\pi}\), where
\[\kappa_{\pi}\left(x_{1},\ldots,x_{n}\right):=\kappa_{k_{1}}\left(x_{c_{1}},\ldots,x_{c_{k_{1}}}\right)\cdots\kappa_{k_{m}}\left(x_{c_{k_{1}+\cdots+k_{m-1}+1}},\ldots,x_{c_{k_{1}+\cdots+k_{m}}}\right)\]
by
\begin{equation}
\label{formula: free moment cumulant}
\alpha_{n}\left(x_{1},\ldots,x_{n}\right)=\sum_{\pi\in S_{\mathrm{nc}}\left(n\right)}\kappa_{\pi}\left(x_{1},\ldots,x_{n}\right)
\end{equation}
or equivalently (by a standard property of M\"{o}bius functions)
\begin{equation}
\label{formula: free cumulant moment}
\kappa_{n}\left(x_{1},\ldots,x_{n}\right):=\sum_{\pi\in S_{\mathrm{nc}}\left(n\right)}\mu_{\mathrm{sd}}\left(\pi,\tau_{n}\right)\alpha_{\pi}\left(x_{1},\ldots,x_{n}\right)\textrm{.}
\end{equation}
\end{definition}
If subalgebras of a free probability space are free, mixed cumulants (those whose arguments include at least one element from each of more than one subalgebra) vanish; see \cite{MR2266879}.

\section{Real second-order cumulants}

\label{section: freeness}

\begin{definition}
Let \(\left(A,\phi_{1},\phi_{2}\right)\) be a real second-order noncommutative space.

For \(p,q=1,2,\ldots\) the second-order moments \(\alpha_{p,q}:A^{p+q}\rightarrow\mathbb{C}\) are multilinear functions defined by
\[\alpha_{p,q}\left(x_{1},\ldots,x_{p},y_{1},\ldots,y_{q}\right):=\varphi_{2}\left(x_{1}\cdots x_{p},y_{1}\cdots y_{q}\right)\textrm{.}\]
For \(\left({\cal U},\pi\right)\in\mathit{PS}^{\prime}\left(p,q\right)\) with nontrivial block \(U\) and \(\left.\pi\right|_{U}=\left(c_{1},\ldots,c_{r}\right)\left(p+d_{1},\ldots,p+d_{s}\right)\), the moment over \(\left({\cal U},\pi\right)\), \(\alpha_{\left({\cal U},\pi\right)}:A^{p+q}\rightarrow\mathbb{C}\), is defined by
\begin{multline*}
\alpha_{\left({\cal U},\pi\right)}\left(x_{1},\ldots,x_{p},y_{1},\ldots,y_{q}\right)\\:=\alpha_{\left.\pi\right|_{\left[p+q\right]\setminus U}}\left(x_{1},\ldots,x_{p},y_{1},\ldots,y_{q}\right)\alpha_{r,s}\left(x_{c_{1}},\ldots,x_{c_{r}},y_{d_{1}},\ldots,y_{d_{s}}\right)\textrm{.}
\end{multline*}

We define the second-order (real) cumulants \(\kappa_{p,q}\left(x_{1},\ldots,x_{p},y_{1},\ldots,y_{q}\right)\) to be multilinear functions, notating, for \(\left({\cal U},\pi\right)\in\mathit{PS}^{\prime}\left(p,q\right)\) as above,
\begin{multline*}
\kappa_{\left({\cal U},\pi\right)}\left(x_{1},\ldots,x_{p},y_{1},\ldots,y_{q}\right)\\:=\kappa_{\left.\pi\right|_{\left[p+q\right]\setminus U}}\left(x_{1},\ldots,x_{p},y_{1},\ldots,y_{q}\right)\kappa_{r,s}\left(x_{c_{1}},\ldots,x_{c_{r}},y_{d_{1}},\ldots,y_{d_{s}}\right)\textrm{,}
\end{multline*}
such that
\begin{multline}
\alpha_{p,q}\left(x_{1},\ldots,x_{p},y_{1},\ldots,y_{q}\right)
\\=:\sum_{\alpha\in S_{\mathrm{ann-nc}}\left(p,q\right)}\kappa_{\alpha}\left(x_{1},\ldots,x_{p},y_{1},\ldots,y_{q}\right)\\+\sum_{\alpha\in S_{\mathrm{ann-nc}}\left(p,q\right)}\kappa_{\alpha}\left(x_{1},\ldots,x_{p},y_{q}^{t},\ldots,y_{1}^{t}\right)\\+\sum_{\substack{\left(U,\alpha\right)\in\mathit{PS}^{\prime}\left(p,q\right)}}\kappa_{\left(U,\alpha\right)}\left(x_{1},\ldots,x_{p},y_{1},\ldots,y_{q}\right)\textrm{;}
\label{formula: moment cumulant}
\end{multline}
or equivalently (see Lemma~\ref{lemma: moment cumulant}),
\begin{multline}
\kappa_{p,q}\left(x_{1},\ldots,x_{p},y_{1},\ldots,y_{q}\right)\\:=2\sum_{\pi\in S_{\mathrm{disc-nc}}\left(p,q\right)}\left[\mu_{\mathrm{sd}}\left(\pi,1\right)+\mu_{\mathrm{sd}}\left(\hat{\pi},1\right)\right]\alpha_{\pi}\left(x_{1},\ldots,x_{p},y_{1},\ldots,y_{q}\right)\\+\sum_{\pi\in S_{\mathrm{ann-nc}}\left(p,q\right)}\mu_{\mathrm{sd}}\left(\pi,1\right)\alpha_{\pi}\left(x_{1},\ldots,x_{p},y_{1},\ldots,y_{q}\right)\\+\sum_{\pi\in S_{\mathrm{ann-nc}}\left(p,q\right)}\mu_{\mathrm{sd}}\left(\pi,1\right)\alpha_{\pi}\left(x_{1},\ldots,x_{p},y_{q}^{t},\ldots,y_{1}^{t}\right)\\+\sum_{\left(U,\pi\right)\in\mathit{PS}^{\prime}\left(p,q\right)}\mu_{\mathrm{sd}}\left(\hat{\pi},1\right)\alpha_{\left(U,\pi\right)}\left(x_{1},\ldots,x_{p},y_{1},\ldots,y_{q}\right)\textrm{.}
\label{formula: cumulant moment}
\end{multline}
It follows from (\ref{formula: free moment cumulant}) and (\ref{formula: moment cumulant}) that for \(\left({\cal U},\pi\right)\in\mathit{PS}^{\prime}\left(p,q\right)\) as above that:
\begin{multline}
\alpha_{\left({\cal U},\pi\right)}\left(x_{1},\ldots,x_{p},y_{1},\ldots,y_{q}\right)\\=\sum_{\substack{\rho\in S_{\mathrm{ann-nc}}\left(p,q\right)\\\Pi\left(\rho\right)\preceq{\cal U}}}\kappa_{\rho}\left(x_{1},\ldots,x_{p},y_{1},\ldots,y_{q}\right)\\+\sum_{\substack{\rho\in S_{\mathrm{ann-nc}}\left(p,q\right)\\\Pi\left(\rho_{\mathrm{op}_{2}}\right)\preceq{\cal U}}}\kappa_{\rho}\left(x_{1},\ldots,x_{p},y_{q}^{t},\ldots,y_{1}^{t}\right)\\+\sum_{\substack{\left({\cal V},\rho\right)\in\mathit{PS}^{\prime}\left(p,q\right)\\{\cal V}\preceq{\cal U}}}\kappa_{\left({\cal V},\rho\right)}\left(x_{1},\ldots,x_{p},y_{1},\ldots,y_{q}\right)
\label{formula: general moment cumulant}
\end{multline}
and from (\ref{formula: free cumulant moment}) and (\ref{formula: cumulant moment}) that
\begin{multline}
\kappa_{\left({\cal U},\pi\right)}\left(x_{1},\ldots,x_{p},y_{1},\ldots,y_{q}\right)\\=2\sum_{\substack{\rho\in S_{\mathrm{disc-nc}}\left(p,q\right)\\\rho\preceq\pi}}\mu_{\mathrm{sd}}\left(\left.\rho\right|_{\left[p+q\right]\setminus U},\left.\pi\right|_{\left[p+q\right]\setminus U}\right)\\\times\left(\mu_{\mathrm{sd}}\left(\left.\rho\right|_{U},\left.1\right|_{U}\right)+\mu_{\mathrm{sd}}\left(\widehat{\left.\rho\right|_{U}},\left.1\right|_{U}\right)\right)\alpha_{\rho}\left(x_{1},\ldots,x_{p},y_{1},\ldots,y_{q}\right)\\+\sum_{\substack{\rho\in S_{\mathrm{ann-n}}\left(p,q\right)\\\Pi\left(\rho\right)\preceq{\cal U}}}\mu_{\mathrm{sd}}\left(\rho,\pi\right)\alpha_{\rho}\left(x_{1},\ldots,x_{p},y_{1},\ldots,y_{q}\right)\\+\sum_{\substack{S\in S_{\mathrm{ann-nc}}\left(p,q\right)\\\Pi\left(\rho_{\mathrm{op}_{2}}\right)\preceq{\cal U}}}\mu_{\mathrm{sd}}\left(\rho_{\mathrm{op}_{2}},\pi\right)\alpha_{\rho}\left(x_{1},\ldots,x_{p},y_{q}^{t},\ldots,y_{1}^{t}\right)\\+\sum_{\substack{\left({\cal V},\rho\right)\in\mathit{PS}^{\prime}\left(p,q\right)\\{\cal V}\preceq{\cal U}}}\mu_{\mathrm{sd}}\left(\hat{\rho},\hat{\pi}\right)\alpha_{\left({\cal V},\rho\right)}\left(x_{1},\ldots,x_{p},y_{1},\ldots,y_{q}\right)\textrm{.}
\label{formula: general cumulant moment}
\end{multline}
\end{definition}

\begin{lemma}
\label{lemma: moment cumulant}
Definitions (\ref{formula: moment cumulant}) and (\ref{formula: cumulant moment}) are equivalent.
\end{lemma}
\begin{proof}
We note that \(\kappa_{p,q}\) appears with coefficient \(1\) in (\ref{formula: moment cumulant}) and \(\alpha_{p,q}\) appears wtih coefficient \(1\) in (\ref{formula: cumulant moment}).  By induction, the expression for \(\kappa_{p,q}\) is determined uniquely by (\ref{formula: moment cumulant}) and that for \(\alpha_{p,q}\) by (\ref{formula: cumulant moment}).  We may thus demonstrate the equivalence by assuming (\ref{formula: moment cumulant}) and showing (\ref{formula: cumulant moment}), from which the converse follows.  We will do so by induction on \(p+q\) (the case \(p=q=1\) is straightforward).  We expand each summand in the right-hand side of (\ref{formula: cumulant moment}) in cumulants according to (\ref{formula: free moment cumulant}) and (\ref{formula: general moment cumulant}).

If \(\pi\in S_{\mathrm{disc-nc}}\left(p,q\right)\), then \(\kappa_{\pi}\left(x_{1},\ldots,x_{p},y_{1},\ldots,y_{q}\right)\) appears in the expansion of each term \(\alpha_{\rho}\left(x_{1},\ldots,x_{p},y_{1},\ldots,y_{q}\right)\) for every \(\rho\in S_{\mathrm{disc-nc}}\left(p,q\right)\) with \(\rho\succeq\pi\) (which has coefficient \(2\left(\mu_{\mathrm{sd}}\left(\rho,1\right)+\mu_{\mathrm{sd}}\left(\hat{\rho},1\right)\right)\)), in the expansion of \(\alpha_{\rho}\left(x_{1},\ldots,x_{p},y_{1},\ldots,y_{q}\right)\) for every \(\rho\in S_{\mathrm{ann-nc}}\left(p,q\right)\) with \(\rho\succeq\pi\) (which has coefficient \(\mu_{\mathrm{sd}}\left(\rho,1\right)\)), and in the form \(\alpha_{\pi_{\mathrm{op}_{2}}}\left(x_{1},\ldots,x_{p},y_{q}^{t},\ldots,y_{1}^{t}\right)\) in the expansion of \(\alpha_{\rho}\left(x_{1},\ldots,x_{p},y_{q}^{t},\ldots,y_{1}^{t}\right)\) for every \(\rho\in S_{\mathrm{ann-nc}}\left(p,q\right)\) with \(\rho\succeq\pi_{\mathrm{op}_{2}}\) (which has coefficient \(\mu_{\mathrm{sd}}\left(\rho,1\right)\)).  It does not appear in the expansion of the \(\alpha_{\left({\cal U},\rho\right)}\).  Noting that, for \(\rho\in S_{\mathrm{sd}}\left(p,q\right)\), \(\mu_{\mathrm{sd}}\left(\rho,1\right)=\mu_{\mathrm{sd}}\left(\rho_{\mathrm{op}_{i}},1\right)\) (see the comment after Definition~\ref{definition: opposite}), the total coefficient is
\begin{multline*}
2\sum_{\substack{\rho\in S_{\mathrm{disc-nc}}\left(p,q\right)\\\rho\succeq\pi}}\left[\mu_{\mathrm{sd}}\left(\rho,1\right)+\mu_{\mathrm{sd}}\left(\hat{\rho},1\right)\right]+\sum_{\substack{\rho\in S_{\mathrm{ann-nc}}\left(p,q\right)\\\rho\succeq\pi}}\mu_{\mathrm{sd}}\left(\rho,1\right)\\+\sum_{\substack{\rho\in S_{\mathrm{ann-nc}}\\\rho_{\mathrm{op}_{2}\succeq\pi}}}\mu_{\mathrm{sd}}\left(\rho,1\right)=\sum_{\rho\in\left[\pi,1\right]}\mu_{\mathrm{sd}}\left(\rho,1\right)+\sum_{\rho\in\left[\pi_{\mathrm{op}_{2}},1\right]}\mu_{\mathrm{sd}}\left(\rho,1\right)=0\textrm{.}
\end{multline*}

If \(\pi\in S_{\mathrm{ann-nc}}\left(p,q\right)\), then \(\kappa_{\pi}\left(x_{1},\ldots,x_{p},y_{1},\ldots,y_{q}\right)\) appears in the expansion of \(\alpha_{\rho}\left(x_{1},\ldots,x_{p},y_{1},\ldots,y_{q}\right)\) for every \(\rho\in S_{\mathrm{ann-nc}}\left(p,q\right)\) with \(\rho\succeq\pi\) (with coefficient \(\mu_{\mathrm{sd}}\left(\rho,1\right)\)), and in the expansion of \(\alpha_{\left({\cal U},\rho\right)}\left(x_{1},\ldots,x_{p},y_{1},\ldots,y_{q}\right)\) for \(\left({\cal U},\rho\right)\) with \({\cal U}\succeq\Pi\left(\pi\right)\) (with coefficient \(\mu_{\mathrm{sd}}\left(\hat{\rho},1\right)\)).  The total coefficient is then
\[\sum_{\substack{\rho\in S_{\mathrm{ann-nc}}\left(p,q\right)\\\rho\succeq\pi}}\mu_{\mathrm{sd}}\left(\rho,1\right)+\sum_{\substack{\rho\in S_{\mathrm{disc-nc}}\left(p,q\right)\\\hat{\rho}\succeq\pi}}\mu_{\mathrm{sd}}\left(\hat{\rho},1\right)=\sum_{\rho\in\left[\pi,1\right]}\mu_{\mathrm{sd}}\left(\rho,1\right)=0\textrm{.}\]
Likewise the sum of the coefficients on \(\kappa_{\pi}\left(x_{1},\ldots,x_{p},y_{q}^{t},\ldots,y_{1}^{t}\right)\) vanishes.

If \(\left({\cal U},\pi\right)\in\mathit{PS}^{\prime}\left(p,q\right)\), then \(\kappa_{\left({\cal U},\pi\right)}\left(x_{1},\ldots,x_{p},y_{1},\ldots,y_{q}\right)\) appears in the expansion of \(\alpha_{\left({\cal V},\rho\right)}\left(x_{1},\ldots,x_{p},y_{1},\ldots,y_{q}\right)\) for all \(\left({\cal V},\rho\right)\in\mathit{PS}^{\prime}\left(p,q\right)\) with \({\cal U}\succeq{\cal V}\) (which has coefficient \(\mu\left(\hat{\rho},1\right)\)).  For any \(\rho\in S_{\mathrm{disc-nc}}\left(p,q\right)\) with \(\rho\succeq\pi\), there is exactly one such \(\left({\cal V},\rho\right)\), with nontrivial block of \({\cal V}\) containing the two cycles of \(\rho\) which contain the two cycles of \(\pi\) in \(U\).  The total coefficient is
\[\sum_{\substack{\rho\in S_{\mathrm{disc-nc}}\left(p,q\right)\\\rho\succeq\pi}}\mu\left(\hat{\rho},1\right)=\sum_{\hat{\rho}\in\left[\hat{\pi},1\right]}\mu\left(\hat{\rho},1\right)=\left\{\begin{array}{ll}1\textrm{,}&\hat{\rho}=1\\0\textrm{,}&\textrm{otherwise}\end{array}\right.\textrm{.}\]

By the comment at the beginning of the proof the converse follows.
\end{proof}

We will make use of the following folklore lemma, which characterises the disc-noncrossing permutations on one cycle which do not connect any element to its neighbour and the annular noncrossing permutations which in addition do not have any single-element cycles:
\begin{lemma}
\begin{enumerate}
	\item Let \(\tau\in S_{n}\) have one cycle, \(n\geq 2\).  If \(\pi\in S_{\mathrm{nc}}\left(\tau\right)\) does not have any cycle containing both \(a\) and \(\tau\left(a\right)\) for any \(a\in\left[n\right]\), then \(\pi\) has at least two single-element cycles.
\label{item: disc singlets}
	\item Let \(\tau\in S_{n}\) have two cycles.  If \(\pi\in S_{\mathrm{ann-nc}}\left(\tau\right)\) does not have any single-element cycles or any cycle containing both \(a\) and \(\tau\left(a\right)\) for \(a\in\left[p,q\right]\), then it is of the form
\[\pi=\left(a,b\right)\left(\tau\left(a\right),\tau^{-1}\left(b\right)\right)\left(\tau^{2}\left(a\right),\tau^{-2}\left(b\right)\right)\cdots\left(\tau^{-1}\left(a\right),\tau\left(b\right)\right)\]
where \(a\) and \(b\) are in different cycles of \(\tau\) (so the two cycles of \(\tau\) must have the same number of elements for such a \(\pi\) to exist).
\label{item: annular spoke}
\end{enumerate}
\label{lemma: singlets}
\end{lemma}
\begin{proof}
\ref{item: disc singlets}: If \(\pi\) has no cycles with more than one element, then we are done.  Otherwise, there is an orbit of \(\pi\) containing some \(a\) and \(\tau^{k}\left(a\right)\) for some minimal \(k\), \(2\leq k\leq n-2\).  Then any element in \(\tau\left(a\right),\ldots,\tau^{k-1}\left(a\right)\) may not share an orbit of \(\pi\) with any element outside this set (noncrossing condition).  If an orbit of such an element contains more than one element, we can repeat the argument with \(k^{\prime}<k\), so there must be a single-element cycle among \(\tau\left(a\right),\ldots,\tau^{k-1}\left(a\right)\).  We may also repeat the argument starting with \(\tau^{k}\left(a\right)\), to find another single-element cycle.

\ref{item: annular spoke}: No cycle of such a \(\pi\) contains two elements from the same cycle of \(\tau\): if \(a\) and \(\tau^{k}\left(a\right)\) share a cycle of \(\pi\), some element \(x\) of \(\tau\left(a\right),\ldots,\tau^{k-1}\left(a\right)\) must share an orbit of \(\pi\) with an element \(y\) of the other cycle of \(\tau\), since it may not be in a single-element cycle (as in part~\ref{item: disc singlets}).  Then \(\left.\pi\right|_{\left[n\right]\setminus\left\{x,y\right\}}\) is noncrossing on \(\lambda_{x,y}\), so there must be a single-element cycle of \(\pi\) among \(\tau^{k+1}\left(a\right),\tau^{k+2}\left(a\right),\ldots,\tau^{-1}\left(a\right)\) (by the reasoning in part~\ref{item: disc singlets}.).  Thus any cycle of \(\pi\) is of the form \(\left(a,b\right)\) where \(a\) and \(b\) are in different cycles of \(\tau\).  Because \(\left.\pi\right|_{\left[n\right]\setminus\left\{a,b\right\}}\) is noncrossing on \(\lambda_{a,b}\), \(\tau\left(a\right)\) and \(\tau^{-1}\left(b\right)\) cannot both be paired with any other elements, so by induction \(\pi\) must be of the form given.
\end{proof}

The following theorem shows that second-order freeness is equivalent to the vanishing of mixed free and second-order free cumulants:
\begin{theorem}
Subalgebras \(A_{1},\ldots,A_{m}\subseteq A\) of second-order real probability space \(A\) are real second-order free if and only if all mixed free cumulants and all mixed real second-order cumulants vanish.
\end{theorem}
\begin{proof}
Assume first that the subalgebras are second-order free.  We demonstrate that mixed cumulants vanish by induction on \(p+q\).  In the base case, where \(p=q=1\) and \(a_{1}\) and \(b_{1}\) are second-order real free,
\[\kappa_{1,1}\left(a_{1},b_{1}\right)=-\alpha_{2}\left(a_{1},b_{1}\right)-\alpha_{2}\left(a_{1},b_{1}^{t}\right)+2\alpha_{1}\left(a_{1}\right)\alpha_{1}\left(b_{1}\right)+\alpha_{1,1}\left(a_{1},b_{1}\right)\textrm{.}\]
Since \(a_{1}\) is free from \(b_{1}\), \(\alpha_{2}\left(a_{1},b_{1}\right)=\alpha_{2}\left(a_{1},b_{1}^{t}\right)=\alpha_{1}\left(a_{1}\right)\alpha_{1}\left(b_{1}\right)\), and by the definition of second-order freeness, \(\alpha_{1,1}\left(a_{1},b_{1}\right)=0\), so \(\kappa_{1,1}\left(a_{1},b_{1}\right)=0\).

We now consider two cases separately: firstly the case where the \(a_{1},\ldots,a_{p}\) and \(b_{1},\ldots,b_{q}\) are either cyclically alternating or have only one term, and secondly the case where this does not hold.

We first assume that the \(a_{1},\ldots,a_{p}\) and \(b_{1},\ldots,b_{q}\) are each either cyclically alternating or consist of a single term.  Then
\begin{multline}
\alpha_{p,q}\left(\mathring{a}_{1},\ldots,\mathring{a}_{p},\mathring{b}_{1},\ldots,\mathring{b}_{q}\right)\\=\sum_{\substack{\left\{i_{1},\ldots,i_{k}\right\}\subseteq\left[p\right]\\\left\{j_{1},\ldots,j_{l}\right\}\subseteq\left[q\right]}}\left(-1\right)^{p+q-k-l}\varphi_{2}\left(a_{i_{1}}\cdots a_{i_{k}},b_{j_{1}}\cdots b_{j_{l}}\right)\\\prod_{i\in\left[p\right]\setminus\left\{i_{1},\ldots,i_{k}\right\}}\varphi_{1}\left(a_{i}\right)\prod_{j\in\left[q\right]\setminus\left\{j_{1},\ldots,j_{l}\right\}}\varphi_{1}\left(b_{j}\right)
\label{formula: centred}
\end{multline}
where we let \(i_{1}<\cdots<i_{k}\) and \(j_{1}<\cdots<j_{l}\).  Considering a given \(I=\left\{i_{1},\ldots,i_{k}\right\}\) and \(J=\left\{j_{1},\ldots,j_{l}\right\}\) and expanding the moment in cumulants, the summand is:
\begin{multline}
\sum_{\substack{\pi\in S_{\mathrm{ann-nc}}\left(p,q\right)\\\pi\left(i\right)=i,i\notin I}}\kappa_{\pi}\left(a_{1},\ldots,a_{p},b_{1},\ldots,b_{q}\right)\\+\sum_{\substack{\pi\in S_{\mathrm{ann-nc}}\left(p,q\right)\\\pi_{\mathrm{op}_{2}}\left(i\right)=i,i\notin I}}\kappa_{\pi}\left(a_{1},\ldots,a_{p},b_{q}^{t},\ldots,b_{1}^{t}\right)\\+\sum_{\substack{\left({\cal V},\pi\right)\in{\cal PS}^{\prime}\left(p,q\right)\\\left\{i\right\}\in{\cal V},i\notin I}}\kappa_{\left({\cal U},\pi\right)}\left(a_{1},\ldots,a_{p},b_{1},\ldots,b_{q}\right)\textrm{.}
\label{formula: cumulant singlets}
\end{multline}
A summand in (\ref{formula: cumulant singlets}) with \(r\) single-element blocks (of the permutation in the case of the first two sums and of the partition in the case of the third sum) will appear in the expansion of any summand in (\ref{formula: centred}) where \(\left[p+q\right]\setminus\left(I\cup J\right)\) is a subset of the cumulant's single-element blocks.  There are \(\binom{r}{s}\) such summands in (\ref{formula: centred}) with \(\left|\left[p+q\right]\setminus\left(I\cup J\right)\right|=s\), which appear with sign \(\left(-1\right)^{s}\).  The total coefficient on the cumulant summand from (\ref{formula: cumulant singlets}) is then \(\sum_{s=0}^{r}\left(-1\right)^{s}\binom{r}{s}=\left(1-1\right)^{r}\), which is \(0\) unless \(r=0\).  So (\ref{formula: centred}) is the sum over cumulants (in the uncentred variables) with no single-element blocks:
\begin{multline}
\alpha_{p,q}\left(\mathring{a}_{1},\ldots,\mathring{a}_{p},\mathring{b}_{1},\ldots,\mathring{b}_{q}\right)\\=\sum_{\substack{\pi\in S_{\mathrm{ann-nc}}\left(p,q\right)\\\pi\left(i\right)\neq i,i\in\left[p+q\right]}}\kappa_{\pi}\left(a_{1},\ldots,a_{p},b_{1},\ldots,b_{q}\right)\\+\sum_{\substack{\pi\in S_{\mathrm{ann-nc}}\left(p,q\right)\\\pi\left(i\right)\neq i,i\in\left[p+q\right]}}\kappa_{\pi}\left(a_{1},\ldots,a_{p},b_{q}^{t},\ldots,b_{1}^{t}\right)\\+\sum_{\substack{\left({\cal U},\pi\right)\in{\cal PS}^{\prime}\left(p,q\right)\\\left\{i\right\}\notin{\cal U},i\in\left[p+q\right]}}\kappa_{\left({\cal U},\pi\right)}\left(a_{1},\ldots,a_{p},b_{1},\ldots,b_{q}\right)\textrm{.}
\label{formula: centred cumulant expansion}
\end{multline}
In any term, if a cycle of \(\pi\) contains more than one element from the same cycle of \(\tau\), there must be a cycle of \(\pi\) containing two cyclically adjacent elements (Lemma~\ref{lemma: singlets}; at least one cycle of \(\tau\) has at least \(2\) elements since \(p+q\geq 1\)).  Such a term therefore contains a mixed free cumulant or free second-order cumulant, and hence vanishes, except possibly the term \(\kappa_{p,q}\left(a_{1},\ldots,a_{p},b_{1},\ldots,b_{q}\right)\) not covered by the induction hypothesis.  In (\ref{formula: second-order freeness}), the terms on the right-hand side are exactly those in (\ref{formula: moment cumulant}) in which \(\pi\) has no single-element cycles and no cycle contains cyclically adjacent elements (in both cases \(p=q\) and \(p\neq q\)).  Equating this expression for \(\alpha_{p,q}\left(\mathring{a}_{1},\ldots,\mathring{a}_{p},\mathring{b}_{1},\ldots,\mathring{b}_{q}\right)\) with (\ref{formula: centred cumulant expansion}), the only term which does not appear on both sides is \(\kappa_{p,q}\left(a_{1},\ldots,a_{p},b_{1},\ldots,b_{q}\right)\), which must therefore vanish.

If at least one of \(a_{1},\ldots,a_{p}\) and \(b_{1},\ldots,b_{q}\) is not cyclically alternating and consists of more than one term, then there is at least one term followed (cyclically) by a term from the same subalgebra.  Without loss of generality, we may assume these are \(b_{q-1}\) and \(b_{q}\).  Rearranging (\ref{formula: cumulant moment}) applied to \(a_{1},\ldots,a_{p}\) and \(b_{1},\ldots,b_{q-2},b_{q-1}b_{q}\):
\begin{multline}
\alpha_{p,q-1}\left(a_{1},\ldots,a_{p},b_{1},\ldots,b_{q-2},b_{q-1}b_{q}\right)\\=\kappa_{p,q-1}\left(a_{1},\ldots,a_{p},b_{1},\ldots,b_{q-2},b_{q-1}b_{q}\right)\\-2\sum_{\pi\in S_{\mathrm{disc-nc}}\left(p,q-1\right)}\left(\mu\left(\pi,1\right)+\mu\left(\hat{\pi},1\right)\right)\alpha_{\pi}\left(a_{1},\ldots,a_{p},b_{1},\ldots,b_{q-2},b_{q-1}b_{q}\right)\\-\sum_{\pi\in S_{\mathrm{ann-nc}}\left(p,q-1\right)}\mu\left(\pi,1\right)\alpha_{\pi}\left(a_{1},\ldots,a_{p},b_{1},\ldots,b_{q-2},b_{q-1}b_{q}\right)\\-\sum_{\pi\in S_{\mathrm{ann-nc}}\left(p,q-1\right)}\mu\left(\pi,1\right)\alpha_{\pi_{\mathrm{op}_{1}}}\left(a_{p}^{t},\ldots,a_{1}^{t},b_{1},\ldots,b_{q-2},b_{q-1}b_{q}\right)\\-\sum_{\substack{\left({\cal U},\pi\right)\in\mathit{SP}^{\prime}\left(p,q-1\right)\\\left({\cal U},\pi\right)\neq\left(1,\tau_{p,q-1}\right)}}\mu\left(\hat{\pi},1\right)\alpha_{\left({\cal U},\pi\right)}\left(a_{1},\ldots,a_{p},b_{1},\ldots,b_{q-2},b_{q-1}b_{q}\right)\textrm{.}
\label{formula: squished moments}
\end{multline}
By the induction hypothesis, the cumulant vanishes.  A moment on the right side with permutation \(\pi\) is equal to a moment of \(p+q\) terms with permutation \(\pi^{\prime}\) constructed by inserting \(p+q\) in the cycle following \(p+q-1\) (and constructing \({\cal U}^{\prime}\in{\cal P}\left(p+q\right)\) by adding \(p+q\) to the same block of \({\cal U}\) as \(p+q-1\)).  This can be inverted by taking the permutation \(\left.\pi^{\prime}\right|_{\left[p+q-1\right]}\) (and \(\left.{\cal U}^{\prime}\right|_{\left[p+q-1\right]}\)).  Noting that \(\pi^{\prime}\left(p+q-1\right)=p+q\) is equivalent to \(\mathrm{Kr}^{-1}\left(\pi^{\prime}\right)\) containing the cycle \(\left(p+q\right)\) and that \(\mathrm{Kr}^{-1}\left(\pi^{\prime}\right)=\mathrm{Kr}^{-1}\left(\pi\right)\cdot\left(p+q\right)\), this transformation is a poset bijection between \(S_{\mathrm{sd}}\left(p,q-1\right)\) and the \(\pi^{\prime}\in S_{\mathrm{sd}}\left(p,q\right)\) for which \(\pi^{\prime}\left(p+q-1\right)=p+q\).  We note that for \(\rho\in S_{\mathrm{sd}}\left(p,q\right)\), there is a \(\tilde{\rho}\in S_{\mathrm{sd}}\left(p,q-1\right)\) such that the \(\pi\in S_{\mathrm{sd}}\left(p,q-1\right)\) with \(\pi^{\prime}\succeq\rho\) are exactly those with \(\pi\succeq\tilde{\rho}\): we have
\[\rho\preceq\pi^{\prime}\Leftrightarrow\widehat{\mathrm{Kr}}^{-1}\left(\rho\right)\succeq\widehat{\mathrm{Kr}}^{-1}\left(\pi^{\prime}\right)\Leftrightarrow\left.\mathrm{Kr}^{-1}\left(\rho\right)\right|_{\left[p+q-1\right]}\cdot \left(p+q\right)\succeq\widehat{\mathrm{Kr}}^{-1}\left(\pi^{\prime}\right)\textrm{.}\]
(We interpret the permutation on the left-hand side of the last inequality as the element of \(S_{\mathrm{sd}}\left(p,q\right)\) covered by \(\widehat{\mathrm{Kr}}^{-1}\left(\rho\right)\).  We verify that this permutation is also noncrossing, since the new cycle contains only one element and therefore cannot introduce any of the nonstandard or crossing conditions; and that \(\mathrm{Kr}^{-1}\left(\pi^{\prime}\right)\) is still noncrossing on this new permutation, as it is noncrossing on each of the disjoint parts on \(\left[p+q-1\right]\) and \(\left\{p+q\right\}\).)  The inequality is then equivalent to
\begin{multline*}
\widehat{\mathrm{Kr}}\left(\left.\mathrm{Kr}^{-1}\left(\rho\right)\right|_{\left[p+q-1\right]}\cdot\left(p+q\right)\right)\preceq\pi^{\prime}\\\Leftrightarrow\tilde{\rho}:=\left.\widehat{\mathrm{Kr}}\left(\left.\mathrm{Kr}^{-1}\left(\rho\right)\right|_{\left[p+q-1\right]}\cdot\left(p+q\right)\right)\right|_{\left[p+q-1\right]}\preceq\pi\textrm{.}
\end{multline*}
Interpreting the moments on the right-hand side of (\ref{formula: squished moments}) as moments over \(\pi^{\prime}\), \(\left(\pi_{\mathrm{op}_{1}}\right)^{\prime}\), and \(\left({\cal U}^{\prime},\pi^{\prime}\right)\) as appropriate, we expand them in cumulants on \(p+q\) terms according to (\ref{formula: free moment cumulant}) and (\ref{formula: moment cumulant}) and collect the coefficient on each cumulant.  For \(\rho\in S_{\mathrm{disc-nc}}\left(p,q\right)\), the total coefficient on \(\kappa_{\rho}\left(a_{1},\ldots,a_{p},b_{1},\ldots,b_{q}\right)\) (including when it appears in the third sum on the right-hand side of (\ref{formula: squished moments}) as \(\kappa_{\rho_{\mathrm{op}_{1}}}\left(a_{p}^{t},\ldots,a_{1}^{t},b_{1},\ldots,b_{q}\right)\)) is
\begin{multline*}
-2\sum_{\substack{\pi\in S_{\mathrm{disc-nc}}\left(p,q-1\right)\\\pi\succeq\tilde{\rho}}}\left(\mu\left(\pi,1\right)+\mu\left(\hat{\pi},1\right)\right)-\sum_{\substack{\pi\in S_{\mathrm{ann-nc}}\left(p,q-1\right)\\\pi\succeq\tilde{\rho}}}\mu\left(\pi,1\right)\\-\sum_{\substack{\pi\in S_{\mathrm{ann-nc}}\left(p,q-1\right)\\\pi_{\mathrm{op}_{1}}\succeq\tilde{\rho}_{\mathrm{op}_{1}}}}\mu\left(\pi,1\right)=0\textrm{,}
\end{multline*}
for \(\rho\in S_{\mathrm{ann-nc}}\left(p,q\right)\), the total coefficient on \(\kappa_{\rho}\left(a_{1},\ldots,a_{p},b_{1},\ldots,b_{q}\right)\) is
\[-\sum_{\substack{\pi\in S_{\mathrm{ann-nc}}\left(p,q-1\right)\\\pi\succeq\tilde{\rho}}}\mu\left(\pi,1\right)-\sum_{\substack{\left({\cal U},\pi\right)\in\mathit{PS}^{\prime}\left(p,q-1\right)\\\left({\cal U},\pi\right)\neq\left(1,\tau_{p,q-1}\right)\\{\cal U}\succeq\Pi\left(\tilde{\rho}\right)}}\mu\left(\hat{\pi},1\right)=\mu\left(1,1\right)=1\]
(and likewise for \(\rho\) with \(\rho_{\mathrm{op}_{1}}\in S_{\mathrm{ann-nc}}\left(p,q\right)\)), and for \(\left({\cal V},\rho\right)\in\mathit{PS}^{\prime}\left(p,q\right)\) with \(\left({\cal V},\rho\right)\neq\left(1,\tau_{p,q}\right)\), the total coefficient on \(\kappa_{\left({\cal V},\rho\right)}\left(a_{1},\ldots,a_{p},b_{1},\ldots,b_{q}\right)\) is
\[-\sum_{\substack{\left({\cal U},\pi\right)\in\mathit{PS}^{\prime}\left(p,q-1\right)\\\left({\cal U},\pi\right)\neq\left(1,\tau_{p,q-1}\right)\\{\cal U}\succeq{\cal V}}}\mu\left(\hat{\pi},1\right)=\mu\left(1,1\right)=1\textrm{.}\]
Substituting this cumulant expansion of (\ref{formula: squished moments}) for \(\alpha_{p,q}\left(a_{1},\ldots,a_{p},b_{1},\ldots,b_{q}\right)\) in (\ref{formula: moment cumulant}), the only term that does not appear on both the left-hand and right-hand side is \(\kappa_{p,q}\left(a_{1},\ldots,a_{p},b_{1},\ldots,b_{q}\right)\), which therefore vanishes.

The converse of the theorem is much simpler.  If mixed cumulants vanish, then we may expand \(\alpha_{p,q}\left(\mathring{a}_{1},\ldots,\mathring{a}_{p},\mathring{b}_{1},\ldots,\mathring{b}_{q}\right)\) as (\ref{formula: moment cumulant}).  Of the first and second sums, over \(S_{\mathrm{ann-nc}}\left(p,q\right)\), all terms with single-element cycles vanish.  Of the terms in the third sum, over \(\left({\cal U},\pi\right)\in\mathit{PS}^{\prime}\left(p,q\right)\), each \(\pi\) has more than two single-element cycles (Lemma~\ref{lemma: singlets}), so not all such cycles may be contained in the nontrivial block of \({\cal U}\), and thus all such terms vanish.  The remaining terms from the first two sums are those characterized by part~\ref{item: annular spoke} of Lemma~\ref{lemma: singlets}, which gives us (\ref{formula: second-order freeness}).
\end{proof}

\section{Vertex cumulants}

\label{section: matrix}

\begin{definition}
Denote the M\"{o}bius function on \({\cal P}\left(I\right)\) by \(\mu_{{\cal P}}\) (\(I\) will be clear from context).  Then
\[\mu_{{\cal P}}\left({\cal U},{\cal V}\right)=\prod_{V\in{\cal V}}\left(-1\right)^{\#\left(\left.{\cal U}\right|_{V}\right)-1}\left(\#\left(\left.{\cal U}\right|_{V}\right)-1\right)!\]
(see \cite{MR2266879}, exercises to Chapter~10).
\end{definition}

\begin{definition}
\label{definition: partition geodesic}
For \({\cal U},{\cal V}\in{\cal P}\left(n\right)\), let \(\Gamma\left({\cal U},{\cal V}\right)\) be the set of partitions \({\cal W}\in{\cal P}\left(n\right)\) such that
\[\#\left({\cal U}\right)-\#\left({\cal U}\vee{\cal W}\right)=\#\left({\cal U}\vee{\cal V}\right)-\#\left({\cal U}\vee{\cal V}\vee{\cal W}\right)\textrm{.}\]
\end{definition}
See, e.g.\ \cite{2015arXiv151101087R} for a proof that for any \({\cal W}\), the left-hand side is greater than or equal to the right-hand side, so the \({\cal W}\in\Gamma\left({\cal U},{\cal V}\right)\) are those maximizing the right-hand side.

\begin{definition}
For random variables \(Y_{1},\ldots,Y_{n}\) and \({\cal U}\in{\cal P}\left(n\right)\), the (classical) moment over \({\cal U}\) is
\[a_{{\cal U}}\left(Y_{1},\ldots,Y_{n}\right):=\prod_{U\in{\cal U}}\mathbb{E}\left(\prod_{k\in U}Y_{k}\right)\textrm{.}\]

For \({\cal U}\in{\cal P}\left(n\right)\), we define the {\em classical cumulants} \(k_{{\cal U}}\) over \(n\) random variables \(Y_{1},\ldots,Y_{n}\) by
\[a_{{\cal U}}\left(Y_{1},\ldots,Y_{n}\right)=:\sum_{{\cal V}\preceq{\cal U}}k_{{\cal V}}\left(Y_{1},\ldots,Y_{n}\right)\]
or equivalently
\[k_{{\cal U}}\left(Y_{1},\ldots,Y_{n}\right):=\sum_{{\cal V}\preceq{\cal U}}\mu_{{\cal P}}\left({\cal V},{\cal U}\right)a_{{\cal V}}\left(Y_{1},\ldots,Y_{n}\right)\textrm{.}\]
For \(n=1,2,\ldots\),
\[k_{n}\left(Y_{1},\ldots,Y_{n}\right):=k_{1_{{\cal P}\left(n\right)}}\left(Y_{1},\ldots,Y_{n}\right)=\sum_{{\cal U}\in{\cal P}\left(n\right)}\mu_{{\cal P}}\left({\cal U},1\right)a_{{\cal U}}\left(Y_{1},\ldots,Y_{n}\right)\textrm{.}\]
\end{definition}

\begin{definition}
\label{definition: premaps}
Throughout, we will denote \(\delta:k\mapsto -k\).

For \(I\subseteq\left[n\right]\), we define \(\mathit{PM}\left(I\right)\) as the set of permutations \(\pi\) on the set \(\left\{k,-k:k\in I\right\}\) such that \(\pi\left(k\right)=-\pi^{-1}\left(-k\right)\) and such that no \(k\) is in the same cycle as \(-k\).  The cycles of such  \(\pi\) consist of pairs where each cycle in a pair may be obtained from the other by reversing the order and the signs on each of the integers.  See, e.g., \cite{MR1813436, 2012arXiv1204.6211R} for an explanation of how elements of \(\mathit{PM}\left(I\right)\) correspond to maps on unoriented surfaces and the topological interpretations of the related definitions.  We denote \(\mathit{PM}\left(\left[n\right]\right)\) by \(\mathit{PM}\left(n\right)\).

For \(\pi\in\mathit{PM}\left(I\right)\), we define \(\Pi\left(\pi\right)\) to be the partition, each of whose blocks corresponds to a pair of cycles described above, containing the absolute values of the integers in that pair of cycles.  To simplify notation, we denote the restriction of \(\pi\) to \(J\cup\left(-J\right)\) by \(\left.\pi\right|_{J}\) (where \(-J=\left\{-k:k\in J\right\}\)).

For \(\pi\in\mathit{PM}\left(I\right)\), we may choose one from each pair of cycles by choosng the cycle where the integer with the lowest absolute value appears with a positive sign.  We define \(\mathit{FD}\left(\pi\right)\subseteq I\) as the set of elements appearing in the cycles thus chosen.  We define \(\pi_{+}:=\left.\pi\right|_{\mathit{FD}\left(I\right)}\) and \(\pi_{-}:=\delta\pi_{+}^{-1}\delta\) (i.e.\ \(\pi_{+}\) is the product of the cycles chosen above, and \(\pi_{-}\) is the product of the cycles which were not chosen).  (When these permutations appear in an matrix expression, typically which cycle of each pair was chosen for \(\pi_{+}\) is arbitrary and does not change the value of the expression.)

For \(\pi,\rho\in\mathit{PM}\left(I\right)\) we define the Kreweras complement of \(\rho\) (relative to \(\pi\)) as
\[\mathrm{Kr}_{\pi}\left(\rho\right):=\pi_{-}^{-1}\rho^{-1}\pi_{+}\textrm{.}\]
For \(\pi\in S\left(I\right)\), we let \(\mathrm{Kr}_{\pi}\left(\rho\right):=\mathrm{Kr}_{\pi\delta\pi^{-1}\delta}\left(\rho\right)\).

We define the Euler characteristic of \(\rho\) (relative to \(\pi\)) by
\[\chi_{\pi}\left(\rho\right):=\#\left(\pi\right)/2+\#\left(\rho\right)/2+\#\left(\mathrm{Kr}_{\pi}\left(\rho\right)\right)/2-n\textrm{.}\]
We say that \(\rho\in\mathit{PM}\left(I\right)\) is noncrossing (on \(\pi\)) if \(\chi_{\pi}\left(\rho\right)=\#\left(\pi\right)\) (see \cite{MR3217665} for a proof that the left-hand side is less than or equal to the right-hand side).  We denote the set of such \(\rho\) by \(\mathit{PM}_{\mathrm{nc}}\left(\pi\right)\).

We define \(\mathit{PPM}\left(\tau\right)\subseteq{\cal P}\left(I\right)\times\mathit{PM}\left(I\right)\) as the ordered pairs \(\left({\cal U},\pi\right)\) with \({\cal U}\succeq\Pi\left(\pi\right)\).  For \(\tau\in S\left(I\right)\), we define \(\mathit{PPM}^{\prime}\left(\tau\right)\subseteq\mathit{PPM}\left(I\right)\) as the \(\left({\cal U},\pi\right)\) where \(\pi\in\mathit{PM}_{\mathrm{nc}}\left(\tau\right)\) and \({\cal U}\in\Gamma\left(\Pi\left(\tau\right),\Pi\left(\pi\right)\right)\).  (Such \({\cal U}\) can be thought of as those that join the smallest number of blocks of \(\Pi\left(\pi\right)\) for a resulting partition \(\Pi\left(\tau\right)\vee{\cal U}\).)  As usual we denote \(\mathit{PPM}\left(\tau_{r_{1},\ldots,r_{k}}\right)\) and \(\mathit{PPM}^{\prime}\left(\tau_{r_{1},\ldots,r_{k}}\right)\) by \(\mathit{PPM}\left(r_{1},\ldots,r_{k}\right)\) and \(\mathit{PPM}^{\prime}\left(r_{1},\ldots,r_{k}\right)\) respectively.

We extend the notation \(f_{\pi}\) from \(S\left(I\right)\) to \(\mathit{PM}\left(I\right)\), for the functions \(f\) for which \(f_{\pi}\) is defined.  We let \(X_{-k}:=X_{k}^{T}\).  Then for \(\pi\in\mathit{PM}\left(I\right)\), we let \(f_{\pi}:=f_{\pi^{+}}\).

We modify this notation to express classical cumulants of traces.  For
\[\pi_{+}=\left(c_{1},\ldots,c_{r_{1}}\right)\cdots\left(c_{r_{1}+\cdots+r_{m-1}+1},\ldots,c_{r_{1}+\cdots+r_{m}}\right)\]
we define
\[k_{\pi}\left(X_{1},\ldots,X_{n}\right)=k_{m}\left(\mathrm{tr}\left(X_{c_{1}}\cdots X_{c_{n_{1}}}\right),\ldots,\mathrm{tr}\left(X_{c_{n_{1}+\cdots+n_{r-1}+1}}\cdots X_{c_{n}}\right)\right)\textrm{.}\]

We extend this notation to \(\left({\cal U},\pi\right)\in\mathit{PPM}\left(n\right)\) by letting
\[f_{\left({\cal U},\pi\right)}\left(X_{1},\ldots,X_{n}\right)=\prod_{U\in{\cal U}}f_{\left.\pi\right|_{U}}\left(X_{1},\ldots,X_{n}\right)\textrm{.}\]
\end{definition}

The following result is shown in \cite{MR3217665}:
\begin{lemma}
\label{lemma: annular premaps}
For \(\tau:=\tau_{p,q}\) with two cycles, the \(\pi\in S_{\mathrm{nc}}\left(p,q\right)\) may be divided into three disjoint sets, which are respectively in bijection with \(S_{\mathrm{disc-nc}}\left(p,q\right)\), \(S_{\mathrm{ann-nc}}\left(p,q\right)\), and \(S_{\mathrm{ann-nc}}\left(p,q\right)\).  For each \(\pi\) in the union of these three sets we will construct a permutation \(\pi^{\prime}\in S_{\mathrm{nc}}\left(p,q\right)\) as described below.

In the first two cases, \(\pi\) never takes any element of \(\left[p+q\right]\) to \(-\left[p+q\right]\).  For such \(\pi\) we let \(\pi^{\prime}:=\left.\pi\right|_{\left[p+q\right]}\).  In the third case, \(\pi^{\prime}\) never takes any element of \(\left[p\right]\cup\left(-\left[p+1,p+q\right]\right)\) to an element of \(\left(-\left[p\right]\right)\cup\left[p+1,p+q\right]\).  Here we construct \(\pi^{\prime}\in S_{p+q}\) from \(\left.\pi\right|_{\left[p\right]\cup\left(-\left[p+1,p+q\right]\right)}\) by replacing elements \(-\left(p+1\right),\ldots,-\left(p+q\right)\) with \(p+q,\ldots,p-1\) respectively.

In each case, the integer partition whose parts are the sizes of blocks in \(\Pi\left(\mathrm{Kr}_{\tau}\left(\pi^{\prime}\right)\right)\) is the same as that whose parts are the sizes of the blocks of \(\Pi\left(\mathrm{Kr}_{\tau}\left(\pi\right)\right)\).
\end{lemma}
See Figure~\ref{figure: moment cumulant} for the topological intuition.  The last case may be interpreted as the planar cases in which the second disc has been ``flipped over'' relative to the first.

\begin{definition}[Collins, \'{S}niady]
The {\em (real) Weingarten function} \(\mathrm{Wg}:{\cal P}\rightarrow\mathbb{R}\) is defined in \cite{MR2217291, MR2567222}.  For \(N>0\) and \({\cal U}_{1},{\cal U}_{2}\in{\cal P}_{2}\left(2n\right)\), we let \(\mathrm{Gr}\left({\cal U}_{1},{\cal U}_{2}\right)=N^{\#\left({\cal U}_{1}\vee{\cal U}_{2}\right)}\).  The Weingarten function \(\mathrm{Wg^{N}}:{\cal P}_{2}\left(2n\right)\rightarrow\mathbb{R}\) is given by
\[\mathrm{Wg}^{N}\left({\cal U}_{1},{\cal U}_{2}\right)=\left[\mathrm{Gr}\left({\cal U}_{1},{\cal U}_{2}\right)\right]_{{\cal U}_{1},{\cal U}_{2}}^{-1}\]
(the right-hand side being interpreted as a matrix inverse of a matrix with indices in \({\cal P}_{2}\left(2n\right)\)).  We will suppress \(n\) and \(N\) in the notation.  We define the Weingarten function \(\mathrm{Wg}:{\cal P}\left(n\right)\rightarrow\mathbb{R}\) to be equal to \(\mathrm{Wg}\left({\cal U}_{1},{\cal U}_{2}\right)\) where each block of \({\cal U}\) corresponds bijectively to a block of \({\cal U}_{1}\vee{\cal U}_{2}\) with exactly twice as many elements.

For large \(N\), \(\mathrm{Wg}\left({\cal U}\right)={\cal O}\left(N^{n-\#\left({\cal U}\right)}\right)\).  We define the normalized Weingarten function \(\mathrm{wg}\left({\cal U}\right):=N^{-n+\#\left({\cal U}\right)}\mathrm{Wg}\left({\cal U}\right)\).

It is possible to define a sort of cumulant of the Weingarten function for \({\cal U},{\cal V},{\cal W}\in{\cal P}\left(n\right)\) with \({\cal U}\preceq{\cal V}\preceq{\cal W}\):
\[\mathrm{wg}_{{\cal U},{\cal V},{\cal W}}:=\sum_{{\cal X}:{\cal V}\preceq{\cal X}\preceq{\cal W}}\mu_{{\cal P}}\left({\cal X},{\cal W}\right)\prod_{X\in{\cal X}}\mathrm{wg}\left(\left.{\cal U}\right|_{X}\right)\]
or equivalently (standard property of M\"{o}bius functions)
\[\mathrm{wg}\left({\cal U}\right)=\sum_{{\cal X}:{\cal V}\preceq{\cal X}\preceq{\cal W}}\prod_{X\in{\cal X}}\mathrm{wg}_{{\cal U},{\cal X},{\cal W}}\textrm{.}\]
In our uses the middle term \({\cal V}\) will typically be equal to the first term \({\cal U}\).  When this is the case, it may be omitted.

Asymptotically
\[\mathrm{wg}_{{\cal U},{\cal V}}=\gamma_{{\cal U},{\cal V}}N^{-2\left(\#\left({\cal U}\right)-\#\left({\cal V}\right)\right)}+{\cal O}\left(N^{-2\left(\#\left({\cal U}\right)-\#\left({\cal V}\right)\right)-1}\right)\textrm{.}\]
See \cite{MR2217291, MR1761777, 2015arXiv151101087R} for more detail.
\end{definition}

The following monomial integration formula is from \cite{MR2217291}; see also \cite{MR2567222}, and \cite{2015arXiv151101087R} for this formulation.
\begin{theorem}[Collins, \'{S}niady]
Let \(O:\Omega\rightarrow M_{N\times N}\) be an \(N\times N\) be a random matrix whose probability measure is a Haar measure (from here on referred to as a Haar-distributed matrix) on the \(N\times N\) orthogonal matrices.  Let \(i^{+},i^{-}:\left[n\right]\rightarrow\left[N\right]\).  Then
\[\mathbb{E}\left(O_{i^{+}\left(1\right),i^{-}\left(1\right)}\cdots O_{i^{+}\left(2n\right),i^{-}\left(2n\right)}\right)=\sum_{\substack{{\cal U}_{+},{\cal U}_{-}\in{\cal P}_{2}\left(2n\right)\\i^{\pm}\left(r\right)=i^{\pm}\left(s\right):\left\{r,s\right\}\in{\cal U}_{\pm}}}\mathrm{Wg}\left({\cal U}_{+},{\cal U}_{-}\right)\textrm{.}\]
\label{theorem: Weingarten}
\end{theorem}

\begin{definition}[Capitaine, Casalis]
If \(X_{1},\ldots,X_{n}:\Omega\rightarrow M_{N\times N}\left(\mathbb{C}\right)\) are random matrices, then the (normalized) matrix cumulants \(c_{\rho}\), \(\rho\in\mathit{PM}\left(n\right)\) are defined by the equations, for \(\pi\in\mathit{PM}\left(n\right)\):
\begin{equation}
\label{formula: matrix moment-cumulant}
\mathbb{E}\left(\mathrm{tr}_{\pi}\left(X_{1},\ldots,X_{n}\right)\right)=:\sum_{\rho\in\mathit{PM}\left(n\right)}N^{\chi_{\pi}\left(\rho\right)-\#\left(\pi\right)}c_{\rho}\left(X_{1},\ldots,X_{n}\right)
\end{equation}
or equivalently,
\begin{multline}
\label{formula: matrix cumulant-moment}
c_{\rho}\left(X_{1},\ldots,X_{n}\right)\\:=\sum_{\pi\in\mathit{PM}\left(n\right)}N^{\chi_{\rho}\left(\pi\right)-\#\left(\rho\right)}\mathrm{wg}\left(\Pi\left(\mathrm{Kr}_{\rho}\left(\pi\right)\right)\right)\mathbb{E}\left(\mathrm{tr}_{\pi}\left(X_{1},\ldots,X_{n}\right)\right)\textrm{.}
\end{multline}
(We will use the term {\em matrix moment} to refer to expressions of the form of the left-hand side of (\ref{formula: matrix moment-cumulant}), which is slightly more general than the usual usage.)

As usual \(c_{r_{1},\ldots,r_{s}}:=c_{\tau_{r_{1},\ldots,r_{s}}}\).  For \(\left({\cal U},\pi\right)\in\mathit{PPM}\left(n\right)\) we define
\[c_{\left({\cal U},\pi\right)}=\left(X_{1},\ldots,X_{n}\right):=\prod_{U\in{\cal U}}c_{\left.\pi\right|_{U}}\left(X_{1},\ldots,X_{n}\right)\textrm{.}\]
\end{definition}

For reference, for \(\pi\in\mathit{PM}\left(n\right)\), we write an expression for the corresponding matrix cumulant with both the Weingarten functions and the matrix moments expanded in cumulants:
\begin{multline}
c_{\pi}\left(X_{1},\ldots,X_{n}\right)\\=\sum_{\substack{\left({\cal V},\rho\right)\in\mathrm{PPM}\left(n\right)\\{\cal W}\succeq\Pi\left(\mathrm{Kr}_{\pi}\left(\rho\right)\right)}}N^{\chi_{\pi}\left(\rho\right)-2\#\left(\Pi\left(\pi\right)\right)}\mathrm{wg}_{\Pi\left(\mathrm{Kr}_{\pi}\left(\rho\right)\right),{\cal W}}k_{\left({\cal V},\rho\right)}\left(X_{1},\ldots,X_{n}\right)\textrm{.}
\label{formula: expanded cumulant}
\end{multline}

We provide a sketch of the proof of the following theorem from \cite{MR2337139} in the notation of this paper in order to highlight the connection of the matrix cumulants with topological expansion.  See also \cite{MR2240781, MR2483727, 2015arXiv151101087R, 2012arXiv1204.6211R, MR3455672} for more detail.  Combinatorial details such as the bijection \({\cal P}_{2}\left(\pm\left[n\right]\right)\rightarrow\mathit{PM}\left(n\right)\) and the correspondence between \({\cal U}\vee{\cal V}\) and \({\cal U}\delta{\cal V}\) may be filled in by the reader, or see \cite{2015arXiv151101087R, MR3455672}.  (Roughly speaking, if, in the top part of Figure~\ref{figure: orthogonally invariant}, we fix a pairing such as the one shown in solid lines, this fixes the topological gluing.  We may calculate the contribution of such a gluing by letting the other pairing (an example is shown by the dotted lines) vary over all pairings, and summing.  The resulting value is the matrix cumulant associated to that gluing.  If subsets of the matrices are independent, they may be treated as different ``colours'' with which the edges associated to the matrices are coloured, and only gluings compatible with the colouring need be considered.  The contribution is the product over the contribution of each colour.  See Example~\ref{example: gluing})

\begin{theorem}[Capitaine, Casalis]
Formulas (\ref{formula: matrix moment-cumulant}) and (\ref{formula: matrix cumulant-moment}) are equivalent.

Let \(s>0\) be an integer, and let \(w:\left[n\right]\rightarrow\left[s\right]\) be a word in the set of ``colours'' \(\left[s\right]\), and denote the partition of \(\left[n\right]\) whose blocks are (nontrivial) preimages of integers in \(\left[s\right]\) by \(\ker\left(w\right):=\left\{w^{-1}\left(k\right)\neq\emptyset:k\in\left[s\right]\right\}\).  Let \(X_{1},\ldots,X_{n}\) be orthogonally invariant random matrices such that \(X_{i}\) is independent from \(X_{j}\) whenever \(w\left(i\right)\neq w\left(j\right)\).  Then for \(\pi\in\mathit{PM}\left(n\right)\),
\[c_{\pi}\left(X_{1},\ldots,X_{n}\right)=\left\{\begin{array}{ll}\prod_{k\in\left[s\right]}c_{\left.\pi\right|_{w^{-1}\left(k\right)}}\left(X_{1},\ldots,X_{n}\right)&\Pi\left(\pi\right)\preceq\ker\left(w\right)\\0\textrm{,}&\Pi\left(\pi\right)\npreceq\ker\left(w\right)\end{array}\right.\textrm{.}\]
\label{theorem: matrix cumulants}
\end{theorem}
\begin{proof}[Sketch of proof]
We begin by writing a monomial integral formula (expected value of the product of arbitrarily indexed entries) for the random matrices in terms of the matrix cumulants as given by (\ref{formula: matrix cumulant-moment}).  Let \(O\) be a Haar-distributed matrix and let \(i:\pm\left[n\right]\rightarrow\left[N\right]\) be defined so the desired indices of matrx \(X_{k}\) are \(i\left(k\right)i\left(-k\right)\).  Then:
\begin{multline*}
\mathbb{E}\left(X^{\left(1\right)}_{i\left(1\right)i\left(-1\right)}\cdots X^{\left(n\right)}_{i\left(n\right)i\left(-n\right)}\right)=\mathbb{E}\left[\left(OX_{1}O^{T}\right)_{i\left(1\right)i\left(-1\right)}\cdots\left(OX_{n}O^{T}\right)_{i\left(n\right)i\left(-n\right)}\right]\\=\sum_{j:\pm\left[n\right]\rightarrow\left[N\right]}\mathbb{E}\left[\prod_{k\in\pm\left[n\right]}O_{i\left(k\right)j\left(k\right)}\right]\mathbb{E}\left[\prod_{k\in\left[n\right]}X^{\left(k\right)}_{j\left(k\right)j\left(-k\right)}\right]\textrm{.}
\end{multline*}
Applying Theorem~\ref{theorem: Weingarten} to the first expected value:
\[\sum_{j:\pm\left[n\right]\rightarrow\left[N\right]}\sum_{\substack{{\cal U},{\cal V}\in{\cal P}_{2}\left(\pm\left[n\right]\right)\\i\left(r\right)=i\left(s\right):\left\{r,s\right\}\in{\cal U}\\j\left(r\right)=j\left(s\right):\left\{r,s\right\}\in{\cal V}}}\mathrm{Wg}\left({\cal U},{\cal V}\right)\mathbb{E}\left[\prod_{k\in\left[n\right]}X^{\left(k\right)}_{j\left(k\right)j\left(-k\right)}\right]\textrm{.}\]

We collect terms corresponding to a fixed \({\cal U}\), while we consider all possible pairings \({\cal V}\) on the other set of indices and sum over the resulting terms).  We note first that summing over all \(j\) satisfying the constraints given by a fixed \({\cal V}\) gives a product of traces of \(X_{k}\):
\[\sum_{\substack{j:\pm\left[n\right]\rightarrow\left[N\right]\\j\left(r\right)=j\left(s\right):\left\{r,s\right\}\in{\cal V}}}\mathbb{E}\left[\prod_{k\in\left[n\right]}X^{\left(k\right)}_{j\left(k\right)j\left(-k\right)}\right]=\mathbb{E}\left(\mathrm{Tr}_{{\cal V}\delta}\left(X_{1},\ldots,X_{n}\right)\right)\]
where we are using \({\cal V}\) to denote the unique \(\Pi^{-1}\left({\cal V}\right)\in S\left(\pm\left[n\right]\right)\).  Secondly, \(\mathrm{Wg}\left({\cal U},{\cal V}\right)=\mathrm{Wg}\left(\Pi\left(\mathrm{Kr}_{{\cal U}\delta}\left({\cal V}\delta\right)\right)\right)\).  Thus, the contribution for a fixed \({\cal U}\) is
\[\prod_{\left\{r,s\right\}\in{\cal U}}\delta_{i\left(r\right),i\left(s\right)}N^{\#\left(\rho\right)/2-n}c_{\rho}\left(X_{1},\ldots,X_{n}\right)\]
so for \(\pi\in\mathit{PM}\left(n\right)\) the matrix moment is
\[\mathbb{E}\left(\mathrm{tr}_{\pi}\left(X_{1},\ldots,X_{n}\right)\right)=N^{-\#\left(\pi\right)/2}\sum_{\substack{i:\pm\left[n\right]\rightarrow\left[N\right]\\i\circ\delta=i\circ\pi}}\mathbb{E}\left(X^{\left(1\right)}_{i\left(1\right),i\left(-1\right)}\cdots X^{\left(n\right)}_{i\left(n\right),i\left(-n\right)}\right)\textrm{,}\]
from which (\ref{formula: matrix moment-cumulant}) follows, since for a given \(\rho\in\mathit{PM}\left(n\right)\) there are \(N^{\#\left(\mathrm{Kr}_{\rho}\left(\pi\right)\right)}\) possible \(i\).

Since there are \(\left|\mathit{PM}\left(n\right)\right|\) matrix moments which may be calculated from the \(\left|\mathit{PM}\left(n\right)\right|\) cumulants, (\ref{formula: matrix moment-cumulant}) also implies (\ref{formula: matrix cumulant-moment}).

To show the multi-matrix formula, we construct \(s\) independent Haar-distributed matrices \(O_{1},\ldots,O_{s}\) and repeat the above calculations conjugating matrix \(X_{k}\) by \(O_{w\left(k\right)}\).  By the hypothesized independence, the monomial integration formula may be written as the product of the corresponding monomial integration formulas for each colour \(k\in\left[s\right]\).  Since the proposed expression for the cumulants then satisfies (\ref{formula: matrix moment-cumulant}), they must be equal to the cumulants as given by (\ref{formula: matrix cumulant-moment}).
\end{proof}

Considering \(\rho\) as representing a surface gluing, each matched pair of cycles of \(\rho\) represents a vertex containing \(X_{k}\) matrices (Figure~\ref{figure: orthogonally invariant}).  However, the contribution of the vertices is not multiplicative.  In the following definition we consider quantities representing the extent to which the cumulants \(c_{\rho}\) deviate from being multiplicative, with a construction parallel to the construction of the classical cumulants.

\begin{definition}
For \(n=1,2,\ldots\) and \(\left({\cal U},\pi\right)\in\mathit{PPM}\left(n\right)\), we define the vertex cumulants \(K_{\left({\cal U},\pi\right)}\) to be multilinear functions of \(n\) random matrices which are multiplicative over the blocks of \({\cal U}\) (i.e.
\[K_{\left({\cal U},\pi\right)}\left(X_{1},\ldots,X_{n}\right)=\prod_{U\in{\cal U}}K_{\left(\left\{U\right\},\left.\pi\right|_{U}\right)}\left(X_{1},\ldots,X_{n}\right)\]
where each \(K_{\left(\left\{U\right\},\left.\pi\right|_{U}\right)}\) is a function of only the \(X_{k}\) with \(k\in U\)) satisfying
\begin{equation}
c_{\pi}\left(X_{1},\ldots,X_{n}\right)=:\sum_{\left({\cal U},\pi\right)\in\mathit{PPM}\left(n\right)}K_{\left({\cal U},\pi\right)}\left(X_{1},\ldots,X_{n}\right)
\label{formula: matrix-vertex}
\end{equation}
or equivalently, for \(\left({\cal U},\pi\right)\in\mathit{PPM}\left(n\right)\):
\begin{equation}
K_{\left({\cal U},\pi\right)}:=\sum_{\substack{\left({\cal V},\pi\right)\in\mathit{PPM}\left(n\right)\\{\cal V}\preceq{\cal U}}}\mu_{{\cal P}}\left({\cal V},{\cal U}\right)c_{\left({\cal V},\pi\right)}\left(X_{1},\ldots,X_{n}\right)
\label{formula: vertex-matrix}
\end{equation}
(the equivalence of (\ref{formula: matrix-vertex}) and (\ref{formula: vertex-matrix}) follows from a standard properties of M\"{o}bius functions).

If \({\cal U}\) is omitted, we will take it to be equal to \(1\in{\cal P}\left(n\right)\):
\[K_{\pi}\left(X_{1},\ldots,X_{n}\right)=\sum_{\left({\cal V},\pi\right)\in\mathit{PPM}\left(n\right)}\mu_{{\cal P}}\left({\cal V},1\right)c_{\left({\cal V},\pi\right)}\left(X_{1},\ldots,X_{n}\right)\textrm{.}\]
We will denote \(K_{\tau_{r_{1},\ldots,r_{k}}}\) by \(K_{r_{1},\ldots,r_{k}}\).
\label{definition: vertex cumulants}
\end{definition}

\begin{example}
\label{example: gluing}
The diagram shown in Figure~\ref{figure: orthogonally invariant} has traces indexed by \(\pi\in\mathit{PM}\left(7\right)\) with
\[\pi=\left(1,2,3\right)\left(-3,-2,-1\right)\left(4,5,6,7\right)\left(-7,-6,-5,-4\right)\textrm{.}\]
The gluing corresponds to the matrix cumulant \(c_{\rho}\left(X_{1},\ldots,X_{7}\right)\) with
\[\rho=\left(1,-4,5\right)\left(-5,4,-1\right)\left(2\right)\left(-2\right)\left(3\right)\left(-3\right)\left(6,-7\right)\left(7,-6\right)\]
(compare vertices in Figure~\ref{figure: vertices}).  The upper diagram shows how this cumulant corresponds to the contributions from the topological expansion of Haar-distributed orthogonal matrices where the pairing \({\cal U}\) on the first indices of \(O\) is (indexed as in Theorem~\ref{theorem: matrix cumulants}):
\[{\cal U}=\left\{\left\{1,-5\right\}\left\{-1,-4\right\}\left\{2,-2\right\}\left\{3,-3\right\}\left\{4,5\right\}\left\{6,7\right\}\left\{-6,-7\right\}\right\}\textrm{.}\]
Note that \({\cal U}=\Pi\left(\rho\delta\right)\).

In dotted lines, an example pairing \({\cal V}\) is shown on the second indices of the \(O\) matrices:
\[{\cal V}=\left\{\left\{1,-5\right\}\left\{-1,2\right\}\left\{-2,-3\right\}\left\{3,-6\right\}\left\{4,-7\right\}\left\{-4,7\right\}\left\{5,6\right\}\right\}\]
which corresponds to a trace of the \(X_{k}\) along permutation
\[\sigma:={\cal V}\delta=\left(1,2,-3,-6,5\right)\left(-5,6,3,-2,-1\right)\left(4,7\right)\left(-7,-4\right)\]
(compare the loops following the dotted lines through \(X_{k}\) vertices).

The Weingarten function depends on the partition
\[{\cal U}\vee{\cal V}=\left\{\left\{1,-5\right\},\left\{-1,2,-2,3,-3,4,-4,5,6,-6,7,-7\right\}\right\}\]
i.e.\ on integer partition \(6,1\) (compare the lengths of the cycles of permutation \(\chi_{\rho}\left(\sigma\right)\) as well as the lengths of the loops following alternately solid and dotted lines through edges containing \(O\) and \(O^{T}\)).

The value of the cumulant is
\begin{align*}
c_{\rho}&=N^{-2}\sum_{{\cal V}\in{\cal P}_{2}\left(\pm\left[7\right]\right)}\mathrm{Wg}\left({\cal U},{\cal V}\right)\mathbb{E}\left(\mathrm{Tr}_{{\cal V}\delta}\left(X_{1},\ldots,X_{7}\right)\right)
\\&=\sum_{\sigma\in\mathit{PM}\left(7\right)}N^{\chi_{\rho}\left(\sigma\right)-8}\mathrm{wg}\left(\Pi\left(\mathrm{Kr}_{\rho}\left(\sigma\right)\right)\right)\mathbb{E}\left(\mathrm{tr}_{\sigma}\left(X_{1},\ldots,X_{7}\right)\right)\textrm{.}
\end{align*}

The two-vertex cumulant between the top-left and bottom-right vertices in Figure~\ref{figure: vertices} is:
\begin{multline*}
K_{\left(1,5,-4\right)\left(6,-7\right)}\left(X_{1},\ldots,X_{7}\right)=K_{3,2}\left(X_{1},X_{5},X_{4}^{T},X_{6},X_{7}^{T}\right)\\=c_{\left(1,5,-4\right)\left(6,-7\right)}\left(X_{1},\ldots,X_{7}\right)-c_{\left(1,5,-4\right)}\left(X_{1},\ldots,X_{7}\right)c_{\left(6,-7\right)}\left(X_{1},\ldots,X_{7}\right)\\=c_{3,2}\left(X_{1},X_{5},X_{4}^{T},X_{6},X_{7}^{T}\right)-c_{3}\left(X_{1},X_{5},X_{4}^{T}\right)c_{2}\left(X_{6},X_{7}^{T}\right)\textrm{.}
\end{multline*}

We calculate
\[\mathrm{Kr}_{\pi}\left(\rho\right)=\pi_{-}^{-1}\rho^{-1}\pi_{+}=\left(1,2,3,5,-6,7,-4\right)\left(4,-7,6,-5,-3,-2,-1\right)\]
and \(\chi_{3,4}\left(\rho\right)=\#\left(\pi\right)/2+\#\left(\rho\right)/2+\#\left(\sigma\right)/2-n=2+4+1-7=0\) (the surface is a Klein bottle, since it is non-orientable) so the contribution of the surface gluing shown in Figure\ref{figure: orthogonally invariant} is \(N^{-2}c_{\rho}\left(X_{1},\ldots,X_{n}\right)\).
\end{example}

We present several lemmas.  Lemma~\ref{lemma: connected diagrams} provides an expression for vertex cumulants in terms of the cumulants of the Weingarten function and the matrices.  Lemma~\ref{lemma: vanishing cumulants} shows the vanishing of the vertex cumulants on independent, general position matrices.  Lemma~\ref{lemma: classical cumulant connected} gives an expression for the cumulant of traces of matrices in terms of the vertex cumulant, as a sum over connected diagrams.

\begin{lemma}
Let \(X_{1},\ldots,X_{n}:\Omega\rightarrow M_{N\times N}\left(\mathbb{C}\right)\) be random matrices with orthogonally invariant joint probability distribution, and \(\left({\cal U},\pi\right)\in PPM\left(n\right)\).  Then
\begin{multline}
K_{\left({\cal U},\pi\right)}\left(X_{1},\ldots,X_{n}\right)\\=\sum_{\substack{\left({\cal V},\rho\right)\in\mathit{PPM}\left(n\right)\\{\cal W}\succeq\Pi\left(\mathrm{Kr}_{\pi}\left(\rho\right)\right)\\\Pi\left(\pi\right)\vee{\cal V}\vee{\cal W}={\cal U}}}N^{\chi_{\pi}\left(\rho\right)-\#\left(\pi\right)}\mathrm{wg}_{\Pi\left(\mathrm{Kr}_{\pi}\left(\rho\right)\right),{\cal W}}k_{\left({\cal V},\rho\right)}\left(X_{1},\ldots,X_{n}\right)\textrm{,}
\label{formula: connected diagrams}
\end{multline}
i.e.\ the sum of the terms in the expansion (\ref{formula: expanded cumulant}) for which the join of the partition formed by the traces, the partition of the cumulants of the Weingarten function, and the partition of the classical cumulants is exactly the specified partition \({\cal U}\).
\label{lemma: connected diagrams}
\end{lemma}
\begin{proof}
We first verify that the quantities presented on the right-hand side of (\ref{formula: connected diagrams}) are multiplicative over the blocks of \({\cal U}\).  If we take the product over \(U\in{\cal U}\) of the conjectured expressions for \(K_{\left(\left\{{\cal U}\right\},\left.\pi\right|_{U}\right)}\left(X_{1},\ldots,X_{n}\right)\), we will have a sum over choices of \(\rho_{U}\in\mathit{PM}\left(U\right)\) for each \(U\in{\cal U}\), i.e.\ over \(\pi\in\mathit{PM}\left(n\right)\).  Noting that the exponent of \(N\) is additive over blocks of \({\cal U}\) and the two cumulants are multiplicative over the blocks of respectively \({\cal W}\) and \({\cal V}\) and therefore over the blocks of \({\cal U}\succeq{\cal W},{\cal V}\), the result will be the right-hand side of (\ref{formula: connected diagrams}) as desired.

Summing the right-hand side of (\ref{formula: connected diagrams}) over all \({\cal U}\succeq\Pi\left(\pi\right)\) gives us (\ref{formula: expanded cumulant}), so the conjectured expressions satisfy (\ref{formula: matrix-vertex}), demonstrating the lemma.
\end{proof}

\begin{lemma}
If \(X_{1},\ldots,X_{n}:\Omega\rightarrow M_{N\times N}\left(\mathbb{C}\right)\) are random matrices, and a subset of the \(X_{k}\) are independent from and orthogonally in general position to the others, then the cumulant vanishes:
\[K_{\pi}\left(X_{1},\ldots,X_{n}\right)=0\textrm{.}\]
\label{lemma: vanishing cumulants}
\end{lemma}
\begin{proof}
Fix \(\pi\in\mathit{PM}\left(n\right)\), and let \({\cal U}=\left\{U_{1},\ldots,U_{r}\right\}=\Pi\left(\pi\right)\).  We construct a generating function for the matrix cumulants restricted to collections of cycles of \(\pi\): for \(I=\left\{i_{1},\ldots,i_{s}\right\}\subseteq\left[r\right]\) we let
\[f_{I}:=c_{\left.\pi\right|_{U_{i_{1}}\cup\cdots\cup U_{i_{s}}}}\left(X_{1},\ldots,X_{n}\right)\]
and let
\[f\left(x_{1},\ldots,x_{r}\right)=1+\sum_{\substack{I=\left\{i_{1},\ldots,i_{s}\right\}\subseteq\left[r\right]\\I\neq\emptyset}}f_{I}x_{i_{1}}\cdots x_{i_{s}}\textrm{.}\]
We now show \(g\left(x_{1},\ldots,x_{r}\right):=\log f\left(x_{1},\ldots,x_{r}\right)\) is a generating function for the vertex cumulants: expanding in the Taylor series of \(\log\left(z\right)\) around \(1\):
\[g\left(x_{1},\ldots,x_{r}\right)=\sum_{k=1}^{\infty}\frac{\left(-1\right)^{k-1}}{k}\left(f\left(x_{1},\ldots,x_{r}\right)-1\right)^{k}\]
the coefficient of \(x_{1}\cdots x_{r}\) is the sum over partitions \({\cal V}=\left\{V_{1},\ldots,V_{s}\right\}\in{\cal P}\left(r\right)\) of (letting \({\cal V}^{\prime}:=\left\{\bigcup_{k\in V}U_{k}:V\in{\cal V}\right\}\)):
\[s!\frac{\left(-1\right)^{s-1}}{s}f_{V_{1}}\cdots f_{V_{s}}=\mu_{{\cal P}}\left({\cal V},1\right)c_{\left({\cal V}^{\prime},\pi\right)}\left(X_{1},\ldots,X_{n}\right)\textrm{.}\]
(\({\cal V}\) representing contributions from the \(k=s\) term of the logarithm in which the \(x_{i}\) whose subscripts are in the same block \(V_{i}\) are from the same \(\left(f\left(x_{1},\ldots,x_{r}\right)-1\right)\) factor: as the term from each factor is distinct, there are \(s!\) terms corresponding to partition \({\cal V}\) in this way). 

By Theorem~\ref{theorem: matrix cumulants}, if for some \(I\subseteq\left[n\right]\) the \(X_{i}\) for \(i\in I\) are independent and orthogonally in general position from the \(X_{i}\), \(i\notin I\), then for any \(\pi\in\mathit{PM}\left(n\right)\) such that \(\Pi\left(\pi\right)\) connects \(I\) and \(\left[n\right]\setminus I\),
\[c_{\rho}\left(X_{1},\ldots,X_{n}\right)=0\]
and if \(\Pi\left(\pi\right)\) which does not connect \(I\) and \(\left[n\right]\setminus I\)
\[c_{\pi}\left(X_{1},\ldots,X_{n}\right)=c_{\left.\pi\right|_{I}}\left(X_{1},\ldots,X_{n}\right)c_{\left.\pi\right|_{\left[n\right]\setminus I}}\left(X_{1},\ldots,X_{n}\right)\]
Thus, if \(f_{1}\left(x_{1},\ldots,x_{r}\right)\) is the generating function for the matrix cumulants of the \(X_{i}\) with \(i\in I\) (where the terms with an \(x_{i}\) with \(i\notin I\) vanish) and \(f_{2}\left(x_{1},\ldots,x_{r}\right)\) is the generating function for the matrix cumulants of the \(X_{i}\) with \(i\notin I\) (where, likewise, terms with \(x_{i}\) where \(i\in I\) vanish), then the generating function for the matrix cumulants of all the \(X_{1},\ldots,X_{r}\) is
\[f\left(x_{1},\ldots,x_{r}\right)=f_{1}\left(x_{1},\ldots,x_{r}\right)f_{2}\left(x_{1},\ldots,x_{r}\right)\textrm{.}\]
Then the generating function for the vertex cumulants is
\[\log\left(f_{1}\left(x_{1},\ldots,x_{r}\right)f_{2}\left(x_{1},\ldots,x_{r}\right)\right)=\log\left(f_{1}\left(x_{1},\ldots,x_{r}\right)\right)+\log\left(f_{2}\left(x_{1},\ldots,x_{r}\right)\right)\]
which has vanishing coefficients on any term with at least one \(x_{i}\) from each of \(I\) and \(\left[r\right]\setminus I\).
\end{proof}

\begin{lemma}
Let \(r_{1},\ldots,r_{m}\) be positive integers, \(n:=r_{1}+\cdots+r_{m}\), and let \(\tau:=\tau_{r_{1},\ldots,r_{m}}\).  Then
\begin{multline}
k_{m}\left(\mathrm{tr}\left(X_{1}\cdots X_{r_{1}}\right),\ldots,\mathrm{tr}\left(X_{r_{1}+\cdots+r_{m-1}+1}\cdots X_{n}\right)\right)\\=\sum_{\substack{\left({\cal U},\pi\right)\in\mathit{PPM}\left(n\right)\\\Pi\left(\tau\right)\vee{\cal U}=1}}N^{\chi_{\tau}\left(\pi\right)-2m}K_{\left({\cal U},\pi\right)}\left(X_{1},\ldots,X_{n}\right)\textrm{.}
\label{formula: moment vertex-cumulant}
\end{multline}
\label{lemma: classical cumulant connected}
\end{lemma}
\begin{proof}
From the definitions:
\begin{align*}
&k_{\left(1,\tau\right)}\left(X_{1},\ldots,X_{n}\right)\\&=\sum_{{\cal U}\succeq\Pi\left(\tau\right)}\mu_{{\cal P}}\left({\cal U},1\right)a_{\left({\cal U},\tau\right)}\left(X_{1},\ldots,X_{n}\right)\\&=\sum_{{\cal U}\succeq\Pi\left(\tau\right)}\mu_{{\cal P}}\left({\cal U},1\right)\sum_{\substack{\left({\cal V},\pi\right)\in\mathit{PPM}\left(n\right)\\{\cal V}\preceq{\cal U}}}N^{\chi_{\tau}\left(\pi\right)-2m}K_{\left({\cal V},\pi\right)}\left(X_{1},\ldots,X_{n}\right)\textrm{.}
\end{align*}
Reversing the order of summation, we get
\[\sum_{\left({\cal V},\pi\right)\in\mathit{PPM}\left(n\right)}N^{\chi_{\tau}\left(\pi\right)-2m}K_{\left({\cal V},\pi\right)}\left(X_{1},\ldots,X_{n}\right)\sum_{{\cal U}\succeq\Pi\left(\tau\right)\vee{\cal V}}\mu_{{\cal P}}\left({\cal U},1\right)\textrm{.}\]
The result follows from the propertes of the M\"{o}bius function.
\end{proof}

\subsection{Asymptotics}

\label{subsection: asymptotics}

We will consider series of random matrices \(\left\{X_{\lambda}^{\left(N\right)}:\Omega\rightarrow M_{N\times N}\left(\mathbb{C}\right)\right\}_{\lambda\in\Lambda,N\in I}\) where \(I\subseteq\mathbb{N}\), and the asymptotic value of expressions when the dimension of the matrix \(N\rightarrow\infty\).  We will typically suppress the superscript \(N\) in the notation.

\begin{definition}
We say that random matrices \(\left\{X_{\lambda}:\Omega\rightarrow M_{N_{k}\times N_{k}}\left(\mathbb{C}\right)\right\}_{\lambda\in\Lambda}\) (for \(N=N_{1}<N_{2}<\ldots\)) have {\em limit distribution} \(\left\{x_{\lambda}\right\}_{\lambda\in\Lambda}\subseteq A\) (where \(\left(A,\varphi_{1}\right)\) is a noncommutative probability space) if, for any \(\lambda_{1},\ldots,\lambda_{n}\in\Lambda\),
\[\lim_{N\rightarrow\infty}\mathbb{E}\left(\mathrm{tr}\left(X_{\lambda_{1}}\cdots X_{\lambda_{n}}\right)\right)=\phi_{1}\left(x_{\lambda_{1}}\cdots x_{\lambda_{n}}\right)\]
while all higher cumulants of normalized traces of random matrices in the algebra generated by \(X_{\lambda}\) vanish as \(N\rightarrow\infty\).

The matrices also have a {\em real \(m\)th-order limit distribution} if \(\left(A,\varphi_{1},\ldots,\varphi_{m}\right)\) is an \(m\)th-order probability space; if there is an involution on \(\Lambda\), \(\lambda\mapsto-\lambda\), such that \(X_{-\lambda}=X_{\lambda}^{T}\); if for every \(m^{\prime}<m\) they have \(m^{\prime}\)th-order limit distribution \(\left(A,\varphi_{1},\ldots,\varphi_{m^{\prime}}\right)\); if, for any positive integers \(r_{1},\ldots,r_{m}\) and \(X_{\lambda_{1}},\ldots,X_{\lambda_{r_{1}+\cdots+r_{m}}}\) with \(\lambda_{1},\ldots,\lambda_{r_{1}+\cdots+r_{m}}\in\Lambda\):
\begin{multline*}
\lim_{N\rightarrow\infty}N^{2m-2}k_{m}\left(\mathrm{tr}\left(X_{\lambda_{1}}\cdots X_{\lambda_{r_{1}}}\right),\ldots,\mathrm{tr}\left(X_{\lambda_{r_{1}+\cdots+r_{m-1}+1}}\cdots X_{\lambda_{r_{1}+\cdots+r_{m}}}\right)\right)
\\=\varphi_{m}\left(x_{\lambda_{1}}\cdots x_{\lambda_{r_{1}}},\ldots,x_{\lambda_{r_{1}+\cdots+r_{m-1}+1}}\cdots x_{\lambda_{r_{1}+\cdots+r_{m}}}\right)\textrm{;}
\end{multline*}
and if for any \(m^{\prime}>m\) and \(X_{1},\ldots,X_{m^{\prime}}\) in the algebra generated by the \(X_{\lambda}\) all the limits
\[\lim_{N\rightarrow\infty}N^{m+m^{\prime}-2}k_{m^{\prime}}\left(\mathrm{tr}\left(X_{1}\right),\ldots,\mathrm{tr}\left(X_{m^{\prime}}\right)\right)=0\textrm{.}\]
\end{definition}

If random matrices have a higher-order limit distribution, the orders of the vertex cumulants decrease with the number of vertices in the following way:
\begin{proposition}
Let \(m\geq 2\), \(r_{1},\ldots,r_{m}\geq 0\), \(n:=r_{1}+\cdots+r_{m}\), and \(\tau:=\tau_{r_{1},\ldots,r_{m}}\).  If random matrices \(X_{1},\ldots,X_{n}:\Omega\rightarrow M_{N\times N}\left(\mathbb{C}\right)\) have an \(m\)th-order limit distribution \(x_{1},\ldots,x_{n}\), then
\begin{multline}
\lim_{N\rightarrow\infty}N^{2m-2}K_{r_{1},\ldots,r_{m}}\left(X_{1},\ldots,X_{m}\right)\\=\sum_{\substack{\left({\cal U},\pi\right)\in\mathit{PPM}^{\prime}\left(r_{1},\ldots,r_{m}\right)\\{\cal V}\succeq\Pi\left(\mathrm{Kr}_{\tau}\left(\pi\right)\right),{\cal V}\in\Gamma\left(\Pi\left(\tau\right)\vee{\cal U},\Pi\left(\mathrm{Kr}_{\tau}\left(\pi\right)\vee{\cal U}\right)\right)\\\Pi\left(\tau\right)\vee{\cal U}\vee{\cal V}=1}}\gamma_{\Pi\left(\mathrm{Kr}_{\tau}\left(\pi\right)\right),{\cal V}}\alpha_{\left({\cal U},\pi\right)}\left(x_{1},\ldots,x_{n}\right)\textrm{.}
\label{formula: higher free cumulant}
\end{multline}
Thus the limit of the vertex cumulant can be expressed as a polynomial in \(\varphi_{1},\ldots,\varphi_{m}\) applied to elements in the algebra generated by \(x_{1},\ldots,x_{n}\).

For \(m^{\prime}>m\)
\begin{equation}
K_{r_{1},\ldots,r_{m^{\prime}}}\left(X_{1},\ldots,X_{m^{\prime}}\right)={\cal O}\left(N^{2-m-m^{\prime}}\right)\textrm{.}
\label{formula: even higher free cumulants}
\end{equation}
\label{proposition: higher free cumulants}
\end{proposition}
\begin{proof}
We first consider (\ref{formula: higher free cumulant}).  Consider a term associated to \(\rho\in\mathrm{PM}\left(n\right)\) in (\ref{formula: connected diagrams}) (we will switch to the notation of (\ref{formula: connected diagrams}) throughout this proof).

The exponent on \(N\) is
\[\chi_{\pi}\left(\rho\right)-2m\leq 2\#\left(\Pi\left(\pi\right)\vee\Pi\left(\rho\right)\right)-2m\textrm{.}\]

An \(s\)th classical cumulant of traces is \({\cal O}\left(N^{2-s-\min\left(s.m\right)}\right)\leq{\cal O}\left(N^{2\left(1-\min\left(s,m\right)\right)}\right)\); the exponent is never positive.  On a block \(I\in\Pi\left(\pi\right)\vee{\cal V}\), if one of the cumulants with indices in \(I\) has \(s\geq m=\#\left(\Pi\left(\pi\right)\right)\), then this cumulant and hence the product of all cumulants in \(I\) is \({\cal O}\left(N^{2\left(1-\#\left(\left.\Pi\left(\pi\right)\right|_{I}\vee\left.\Pi\left(\rho\right)\right|_{I}\right)\right)}\right)\).  If all cumulants with indices in \(I\) have \(s\leq m\), then by the comment after Definition~\ref{definition: partition geodesic}, we have \(\#\left(\left.{\cal V}\right|_{I}\right)-\#\left(\left.\Pi\left(\rho\right)\right|_{I}\right)\leq\#\left(\left.\Pi\left(\pi\right)\right|_{I}\vee\left.{\cal V}\right|_{I}\right)-\#\left(\left.\Pi\left(\pi\right)\right|_{I}\vee\left.\Pi\left(\rho\right)\right|_{I}\right)\), so again the product of cumulants is \({\cal O}\left(N^{2\left(1-\#\left(\left.\Pi\left(\pi\right)\right|_{I}\vee\left.\Pi\left(\rho\right)\right|_{I}\right)\right)}\right)\).  Multiplying over the blocks \(I\), we get
\[k_{\left({\cal V},\rho\right)}\left(X_{1},\ldots,X_{n}\right)={\cal O}\left(N^{-2\left(\#\left(\Pi\left(\pi\right)\vee\Pi\left(\rho\right)\right)-\#\left(\Pi\left(\pi\right)\vee{\cal V}\right)\right)}\right)\textrm{.}\]

Finally,
\begin{multline*}
\mathrm{wg}_{\Pi\left(\mathrm{Kr}_{\pi}\left(\rho\right)\right),{\cal W}}={\cal O}\left(N^{-2\#\left(\Pi\left(\mathrm{Kr}_{\pi}\left(\rho\right)\right)-\#\left({\cal W}\right)\right)}\right)\\={\cal O}\left(N^{-2\left(\#\left(\Pi\left(\pi\right)\vee{\cal V}\right)-\#\left(\Pi\left(\pi\right)\vee{\cal V}\vee{\cal W}\right)\right)}\right)\textrm{.}
\end{multline*}

Multiplying the three displayed formulas, the limit (\ref{formula: higher free cumulant}) exists.  Surviving terms are those where each factor is of maximal order.  Any classical cumulant on more than \(m\) terms must correspond to a block of \({\cal V}\) which contains at least two blocks \(I_{1},I_{2}\in\Pi\left(\rho\right)\) containing elements from the same block of \(\Pi\left(\pi\right)\vee\Pi\left(\rho\right)\).  If we construct \({\cal V}^{\prime}\) from \({\cal V}\) by removing the elements of \(I_{1}\) from their block and adding block \(I_{1}\), we find we have another term appearing in (\ref{formula: connected diagrams}) (since \(\Pi\left(\pi\right)\vee{\cal V}^{\prime}=\Pi\left(\pi\right)\vee{\cal V}\)) which is higher-order in \(N\), so no surviving term may contain a classical cumulant on more than \(m\) elements.  The first two conditions under the summation in (\ref{formula: higher free cumulant}) follow from the maximality.  Furthermore, since there are no classical cumulants on more than \(m\) terms, (\ref{formula: higher free cumulant}) is a polynomial in \(\phi_{1},\ldots,\phi_{m}\) as desired.

We now consider (\ref{formula: even higher free cumulants}).  Again, let \(I\in\Pi\left(\pi\right)\vee{\cal V}\) and let the blocks of \({\cal V}\) in \(I\) contain \(s_{1},\ldots,s_{t}\) blocks of \(\Pi\left(\rho\right)\) respectively.  By the comment after Definition~\ref{definition: partition geodesic},
\[\sum_{i=1}^{t}\left(1-\min\left(s_{i},m\right)\right)\leq 1-\min\left(\#\left(\left.\Pi\left(\pi\right)\vee\Pi\left(\rho\right)\right|_{I}\right),m\right)\]
(considering the case where each summand on the left-hand side is equal to \(1-s_{i}\) and the case where at least one summand is equal to \(1-m\)).  Summing over the blocks \(I\in\Pi\left(\pi\right)\vee{\cal V}\), we have
\[\sum_{V\in{\cal V}}\left(1-\min\left(\#\left(\left.\rho\right|_{V}\right),m\right)\right)\leq\#\left(\Pi\left(\pi\right)\vee{\cal V}\right)-\min\left(\#\left(\Pi\left(\pi\right)\vee\Pi\left(\rho\right)\right),m\right)\textrm{.}\]
The product of classical cumulants is then
\begin{multline*}
k_{\left({\cal V},\rho\right)}\left(X_{1},\ldots,X_{n}\right)={\cal O}\left(N^{2\#\left({\cal V}\right)-\#\left(\rho\right)-\sum_{V\in{\cal V}}\min\left(\#\left(\left.\rho\right|_{V}\right),m\right)}\right)\\={\cal O}\left(N^{-\left(2\#\left(\Pi\left(\pi\right)\vee{\cal V}\right)-\#\left(\Pi\left(\pi\right)\vee\Pi\left(\rho\right)\right)-\min\left(\#\left(\Pi\left(\pi\right)\vee\Pi\left(\rho\right)\right),m\right)\right)}\right)\textrm{.}
\end{multline*}
By the arguments above, if the \(\min\) function takes the value \(\#\left(\Pi\left(\pi\right)\vee\Pi\left(\rho\right)\right)\) the result follows, and if it take the value \(m\), then
\[K_{r_{1},\ldots,r_{m^{\prime}}}\left(X_{1},\ldots,X_{n}\right)={\cal O}\left(N^{2-2m+\left(m-\#\left(\Pi\left(\pi\right)\vee\Pi\left(\rho\right)\right)\right)}\right)\textrm{.}\]
Since \(\#\left(\Pi\left(\pi\right)\vee\Pi\left(\rho\right)\right)\leq\#\left(\pi\right)=m^{\prime}\), the result follows.
\end{proof}

We can define \(m\)th order free cumulants on \(m\)th order probability spaces by (\ref{formula: higher free cumulant}).  (We will see in Proposition~\ref{proposition: two-vertex limit} that this is consistent with the defintion of \(\kappa_{p,q}\).)  Since the higher-order free cumulants are the limits of vertex cumulants, they vanish when applied to the limits of matrices which are independent and orthogonally in general position.
\begin{definition}
For \(r_{1},\ldots,r_{m}\geq 1\), \(\tau:=\tau_{r_{1},\ldots,r_{m}}\), and \(n:=r_{1}+\cdots+r_{m}\), we define a function \(\kappa_{r_{1},\ldots,r_{m}}:A^{n}\rightarrow\mathbb{C}\)
\begin{multline*}
\kappa_{r_{1},\ldots,r_{m}}\left(x_{1},\ldots,x_{n}\right)\\:=\sum_{\substack{\left({\cal U},\pi\right)\in\mathit{PPM}^{\prime}\left(\tau\right)\\{\cal V}\succeq\Pi\left(\mathrm{Kr}_{\tau}\left(\pi\right)\right),{\cal V}\in\Gamma\left(\Pi\left(\tau\right)\vee{\cal U},\Pi\left(\mathrm{Kr}_{\tau}\left(\pi\right)\right)\right)\\\Pi\left(\tau\right)\vee{\cal U}\vee{\cal V}=1}}\gamma_{\Pi\left(\mathrm{Kr}_{\tau}\left(\pi\right)\right),{\cal V}}\alpha_{\left({\cal U},\pi\right)}\left(x_{1},\ldots,x_{n}\right)\textrm{.}
\end{multline*}
We extend the definition to \(\kappa_{\pi}\) for \(\pi\in\mathit{PM}\left(\tau\right)\) and \(\kappa_{\left({\cal U},\pi\right)}\) for \(\left({\cal U},\pi\right)\in\mathit{PPM}^{\prime}\) as in Definition~\ref{definition: premaps}.
\label{definition: higher-order cumulants}
\end{definition}

The higher-order moments may be expressed as sums of the higher-order cumulants:
\begin{proposition}
Let \(X_{1},\ldots,X_{n}:\Omega\rightarrow M_{N\times N}\left(\mathbb{C}\right)\) be random matrices with an \(m\)th order limit distribution \(x_{1},\ldots,x_{n}\).  Then for \(r_{1}+\cdots+r_{m}=n\), 
\[\alpha_{r_{1},\ldots,r_{m}}\left(x_{1},\ldots,x_{n}\right)=\sum_{\left({\cal U},\pi\right)\in\mathit{PPM}^{\prime}\left(n\right)}\kappa_{\left({\cal U},\pi\right)}\left(x_{1},\ldots,x_{n}\right)\textrm{.}\]
\label{proposition: limit cumulant}
\end{proposition}
\begin{proof}
Let \(\tau:=\tau_{r_{1},\ldots,r_{m}}\), so \(\alpha_{r_{1},\ldots,r_{m}}\left(x_{1},\ldots,x_{n}\right)=\lim_{N\rightarrow\infty}k_{\tau}\left(X_{1},\ldots,X_{n}\right)\).
We expand each cumulant as in Lemma~\ref{lemma: classical cumulant connected} and take the \(N\rightarrow\infty\) limit.  The exponent on \(N\) is at most \(2\#\left(\Pi\left(\tau\right)\vee\Pi\left(\pi\right)\right)-2m\).  By Proposition~\ref{proposition: higher free cumulants}, the product of vertex cumulants is \({\cal O}\left(N^{-2\left(\#\left(\Pi\left(\pi\right)\right)-\#\left({\cal V}\right)\right)}\right)\), and by the comment after Definition~\ref{definition: partition geodesic}, \(-2\left(\#\left(\Pi\left(\pi\right)\right)-\#\left({\cal V}\right)\right)\geq -2\left(\#\left(\Pi\left(\tau\right)\vee\Pi\left(\pi\right)\right)-\#\left(\Pi\left(\tau\right)\vee{\cal V}\right)\right)\).  Thus the limit of each summand exists, and the surviving terms are those where \(\chi_{\tau}\left(\pi\right)=2\#\left(\Pi\left(\tau\right)\vee\Pi\left(\pi\right)\right)\) (so \(\pi\in\mathit{PM}_{\mathrm{nc}}\left(\tau\right)\)) and \(\#\left(\Pi\left(\pi\right)\right)-\#\left({\cal V}\right)=\#\left(\Pi\left(\tau\right)\vee\Pi\left(\pi\right)\right)-\#\left(\Pi\left(\tau\right)\vee{\cal V}\right)\) (so \({\cal U}\in\Gamma\left(\Pi\left(\pi\right),\Pi\left(\tau\right)\vee\Pi\left(\pi\right)\right)\)).  By Definition~\ref{definition: higher-order cumulants}, the limit of the surviving terms is as given.
\end{proof}

\subsection{Second-order freeness}

\label{subsection: second-order}

The limits of two-vertex cumulants \(K_{p,q}\) are second-order free cumulants \(\kappa_{p,q}\).  The definition of \(\kappa_{p,q}\) given in Section~\ref{section: freeness} thus coincides with that given in Section~\ref{section: matrix}.

\begin{proposition}
If random matrices \(X_{1},\ldots,X_{p+q}:\Omega\rightarrow\mathbb{C}\) have a second-order limit distribution \(x_{1},\ldots,x_{p+q}\), then
\begin{multline*}
\lim_{N\rightarrow\infty}N^{2}K_{p,q}\left(X_{1},\ldots,X_{p};X_{p+1},\ldots,X_{p+q}\right)\\=\kappa_{p,q}\left(x_{1},\ldots,x_{p};x_{p+1},\ldots,x_{p+q}\right).
\end{multline*}
\label{proposition: two-vertex limit}
\end{proposition}
\begin{proof}
Let \(\tau:=\tau_{p,q}\).  A summand in (\ref{formula: higher free cumulant}) corresponds to a \(\pi\in\mathrm{PM}_{\mathrm{nc}}\left(\tau\right)\), which corresponds to a \(\pi^{\prime}\) according to Lemma~\ref{lemma: annular premaps}.

If \(\Pi\left(\pi^{\prime}\right)\) does not connect \(\left[p\right]\) and \(\left[p+1,p+q\right]\), then either \({\cal U}=\Pi\left(\pi^{\prime}\right)\) (in which case \({\cal V}\) connects two blocks of \(\Pi\left(\mathrm{Kr}_{\tau}\left(\pi^{\prime}\right)\right)\)), or \({\cal U}\) joins two blocks of \(\Pi\left(\pi^{\prime}\right)\) (in which case \({\cal V}=\Pi\left(\mathrm{Kr}_{\tau}\left(\pi^{\prime}\right)\right)\)).  The terms corresponding to the former are those in the first sum of (\ref{formula: cumulant moment}), if we sum over all such \({\cal V}\).  The terms corresponding to the latter are those in the fourth sum in (\ref{formula: cumulant moment}).

If \(\Pi\left(\pi\right)\) connects \(\left[p\right]\) and \(\left[p+1,p+q\right]\), then \({\cal U}=\Pi\left(\pi\right)\) and \({\cal V}=\Pi\left(\mathrm{Kr}_{\tau}\left(\pi\right)\right)\).  Such terms correspond to those in the second and third sums of (\ref{formula: cumulant moment}).
\end{proof}

We show an example of the types of terms that show up in asymptotic expansion of a two-vertex cumulant.  The limits of these terms appear as terms in the expression for the second-order cumulant.
\begin{example}
We consider the two-vertex cumulant
\[K_{3,5}\left(X_{1},X_{2},X_{3},X_{4},X_{5},X_{6},X_{7},X_{8}\right)\]
and the terms in its expansion corresponding to the four diagrams in Figure~\ref{figure: cumulant moment}.

The first diagram corresponds to
\[\pi=\left(1,3\right)\left(-3,-1\right)\left(2\right)\left(-2\right)\left(4,7\right)\left(-7,-4\right)\left(5\right)\left(-5\right)\left(6\right)\left(-6\right)\left(8\right)\left(-8\right)\]
which has \(\Pi\left(\mathrm{Kr}_{3,5}\left(\pi\right)\right)=\left\{\left\{1,2\right\},\left\{3\right\},\left\{4,5,6\right\},\left\{7,8\right\}\right\}\).  This diagram is not connected, so at least one of \({\cal U}\) and \({\cal V}\) must connect the cycles of \(\tau_{3,5}\).  We consider the term where \({\cal U}=\Pi\left(\pi\right)=\left\{\left\{1,3\right\},\left\{2\right\},\left\{4,7\right\},\left\{5\right\}\left\{6\right\},\left\{8\right\}\right\}\) and the unique nontrivial block of \({\cal V}\) contains the two cycles of the complement traced with dotted lines:
\[{\cal V}=\left\{\left\{1,2\right\},\left\{3,4,5,6\right\},\left\{7,8\right\}\right\}\textrm{.}\]
The contribution of the Weingarten function is
\[\mathrm{wg}_{\Pi\left(\mathrm{Kr}_{\pi}\left(\rho\right)\right),{\cal V}}=\mathrm{wg}\left(\left[2\right]\right)\left[\mathrm{wg}\left(\left[3,1\right]\right)-\mathrm{wg}\left(\left[3\right]\right)\cdot\mathrm{wg}\left(\left[1\right]\right)\right]\mathrm{wg}\left(\left[2\right]\right)\textrm{.}\]
Since \(\gamma_{\Pi\left(\mathrm{Kr}_{\pi}\left(\rho\right)\right),{\cal V}}=\left(-1\right)\cdot 30\cdot\left(-1\right)=30\), this quantity is \(30N^{-2}+{\cal O}\left(N^{-3}\right)\).
The contribution of the moments is
\[\mathbb{E}\left(\mathrm{tr}\left(X_{1}X_{3}\right)\mathrm{tr}\left(X_{2}\right)\mathrm{tr}\left(X_{4}X_{7}\right)\mathrm{tr}\left(X_{5}\right)\mathrm{tr}\left(X_{6}\right)\mathrm{tr}\left(X_{8}\right)\right)\]
which has \(N\rightarrow\infty\) limit
\[\alpha_{2}\left(x_{1},x_{3}\right)\alpha_{1}\left(x_{2}\right)\alpha_{2}\left(x_{4},x_{7}\right)\alpha_{1}\left(x_{5}\right)\alpha_{1}\left(x_{6}\right)\alpha_{1}\left(x_{8}\right)=\alpha_{\pi^{\prime}}\left(x_{1},\ldots,x_{8}\right)\textrm{.}\]
Because \(\chi_{\tau}\left(\pi\right)=4\), the limit (multiplied by \(N^{2}\)) survives, and in the limit its contribution is \(30\alpha_{\pi^{\prime}}\left(x_{1},\ldots,x_{8}\right)\).

The second diagram corresponds to
\[\pi=\left(1,4,6,3\right)\left(-3,-6,-4,-1\right)\left(2,7\right)\left(-7,-2\right)\left(5\right)\left(-5\right)\left(8\right)\left(-8\right)\]
with \(\chi_{3,5}\left(\pi\right)=2\).  An asymptotically nonvanishing term must have \({\cal U}=\Pi\left(\pi\right)\) and \({\cal V}=\Pi\left(\mathrm{Kr}_{\tau_{3,5}}\left(\pi\right)\right)\).  The limit of the cumulant of traces is
\begin{multline*}
\lim_{N\rightarrow\infty}k_{\left({\cal U},\pi\right)}\left(X_{1},\ldots,X_{8}\right)\\=\lim_{N\rightarrow\infty}\mathbb{E}\left(\mathrm{tr}\left(X_{1}X_{4}X_{6}X_{3}\right)\right)\mathbb{E}\left(\mathrm{tr}\left(X_{2}X_{7}\right)\right)\mathbb{E}\left(\mathrm{tr}\left(X_{5}\right)\right)\mathbb{E}\left(\mathrm{tr}\left(X_{8}\right)\right)\\=\alpha_{4}\left(x_{1},x_{4},x_{6},x_{3}\right)\alpha_{2}\left(x_{2},x_{7}\right)\alpha_{1}\left(x_{5}\right)\alpha_{1}\left(x_{8}\right)=\alpha_{\pi^{\prime}}\left(x_{1},\ldots,x_{8}\right)\textrm{.}
\end{multline*}
The limit of the Weingarten function is \(2\), so the limit of the term is \(2\alpha_{\pi^{\prime}}\left(x_{1},\ldots,x_{8}\right)\).

The third diagram corresponds to
\[\pi=\left(1,2,-6\right)\left(6,-2,-1\right)\left(3,-5\right)\left(5,-3\right)\left(4,7\right)\left(-7,-4\right)\left(8\right)\left(-8\right)\]
which has
\[\pi^{\prime}=\left(1,2,6\right)\left(3,7\right)\left(4\right)\left(5,8\right)\in S_{\mathrm{ann-nc}}\left(3,5\right)\textrm{.}\]
The contribution of the moment is
\begin{multline*}
\alpha_{\pi^{\prime}}\left(x_{1},x_{2},x_{3},x_{8}^{t},x_{7}^{t},x_{6}^{t},x_{5}^{t},x_{4}^{t}\right)\\=\alpha_{3}\left(x_{1},x_{2},x_{6}^{t}\right)\alpha_{2}\left(x_{3},x_{5}^{t}\right)\alpha_{2}\left(x_{4},x_{7}\right)\alpha_{1}\left(x_{8}\right)\textrm{.}
\end{multline*}
The limit of this term is \(2\alpha_{3}\left(x_{1},x_{2},x_{6}^{t}\right)\alpha_{2}\left(x_{3},x_{5}^{t}\right)\alpha_{2}\left(x_{4},x_{7}\right)\alpha_{1}\left(x_{8}\right)\).

The fourth diagram corresponds to
\[\pi=\left(1\right)\left(-1\right)\left(2,3\right)\left(-3,-2\right)\left(4,6,7\right)\left(-7,-6,-4\right)\left(5\right)\left(-5\right)\left(8\right)\left(-8\right)\]
Because this diagram is not connected, one of \({\cal U}\) and \({\cal V}\) must connect the two cycles of \(\tau_{3,5}\).  We consider the term where the unique nontrivial block of \({\cal U}\) contains the two lighter grey cycles of \(\pi\):
\[{\cal U}=\left\{\left\{1\right\},\left\{2,3,4,6,7\right\},\left\{5\right\},\left\{8\right\}\right\}\]
while \({\cal V}=\Pi\left(\mathrm{Kr}_{\tau_{3,5}}\left(\pi\right)\right)\).  The cumulant of traces is
\begin{multline*}
k_{\left({\cal U},\pi\right)}\left(X_{1},\ldots,X_{8}\right)\\=\mathbb{E}\left(\mathrm{tr}\left(X_{1}\right)\right)\mathrm{cov}\left(\mathrm{tr}\left(X_{2}X_{3}\right),\mathrm{tr}\left(X_{4}X_{6}X_{7}\right)\right)\mathbb{E}\left(\mathrm{tr}\left(X_{5}\right)\right)\mathbb{E}\left(\mathrm{tr}\left(X_{8}\right)\right)\\=\varphi_{1}\left(x_{1}\right)\varphi_{2}\left(x_{2}x_{3},x_{4}x_{6}x_{7}\right)\varphi_{1}\left(x_{5}\right)\varphi_{1}\left(x_{8}\right)\\=\alpha_{1}\left(x_{1}\right)\alpha_{2,3}\left(x_{2},x_{3},x_{4},x_{6},x_{7}\right)\alpha_{1}\left(x_{5}\right)\alpha_{1}\left(x_{8}\right)=\alpha_{\left({\cal U},\pi\right)}\left(x_{1},\ldots,x_{8}\right)
\end{multline*}
and the limit of the Weingarten function is \(-1\), so the limit of the term is
\[-\alpha_{1}\left(x_{1}\right)\alpha_{2,3}\left(x_{2},x_{3},x_{4},x_{6},x_{7}\right)\alpha_{1}\left(x_{5}\right)\alpha_{1}\left(x_{8}\right)=-\alpha_{\left({\cal U},\pi\right)}\left(x_{1},\ldots,x_{8}\right)\textrm{.}\]
\end{example}

\begin{figure}
\label{figure: cumulant moment}
\centering
\begin{tabular}{cc}
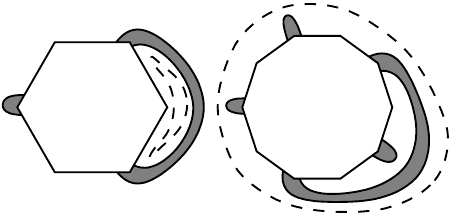\\
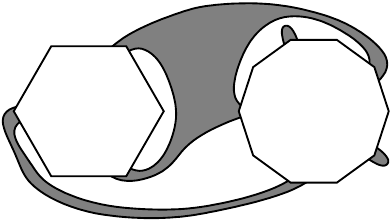\\
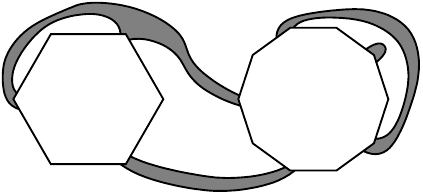\\
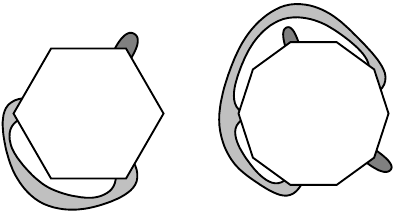
\end{tabular}
\caption{Examples of the four types of diagrams contributing to the highest order of the two-vertex cumulant \(K_{3,5}\left(X_{1},\ldots,X_{8}\right)\).}
\end{figure}

We show an example of the types of terms which appear in the asymptotic expansion of a second-order moment in cumulants:
\begin{example}
\label{example: moment cumulant}
We consider random matrices \(X_{1},\ldots,X_{7}:\Omega\rightarrow M_{N\times N}\left(\mathbb{C}\right)\) with a second-order limit distribution \(x_{1},\ldots,x_{7}\in A\).  We consider the rescaled large \(N\) limit of the covariance of traces:
\begin{multline*}
\lim_{N\rightarrow\infty}N^{2}k_{2}\left(\mathrm{tr}\left(X_{1}X_{2}X_{3}X_{4}\right),\mathrm{tr}\left(X_{5}X_{6}X_{7}\right)\right)=\alpha_{4,3}\left(x_{1},x_{2},x_{3},x_{4};x_{5},x_{6},x_{7}\right)\\=\varphi_{2}\left(x_{1}x_{2}x_{3}x_{4},x_{5}x_{6}x_{7}\right)\textrm{.}
\end{multline*}
The terms in the expansion (\ref{formula: moment vertex-cumulant}) which survive as \(N\rightarrow\infty\) fall into three categories, corresponding to the three sums in (\ref{formula: moment cumulant}), represented by the three examples shown in Figure~\ref{figure: moment cumulant}.

The first diagram is a representative of the diagrams which are annular noncrossing on \(\tau_{p,q}=\left(1,2,3,4\right)\left(5,6,7\right)\).  It corresponds to
\[\pi=\left(1,3,5,6\right)\left(-6,-5,-3,-1\right)\left(2\right)\left(-2\right)\left(4,7\right)\left(-7,-4\right)\]
and \({\cal U}=\Pi\left(\pi\right)=\left\{\left\{1,3,5,6\right\},\left\{2\right\},\left\{4,7\right\}\right\}\), corresponding to term
\[K_{\left({\cal U},\pi\right)}=K_{4}\left(X_{1},X_{3},X_{5},X_{6}\right)K_{1}\left(X_{2}\right)K_{2}\left(X_{4},X_{7}\right)\textrm{.}\]
The \(N\rightarrow\infty\) limit is
\[\kappa_{4}\left(x_{1},x_{3},x_{5},x_{6}\right)\kappa_{1}\left(x_{2}\right)\kappa_{2}\left(x_{4},x_{7}\right)\\=\kappa_{\pi^{\prime}}\left(x_{1},\ldots,x_{7}\right)\]
a term in the free cumulant expansion of \(\alpha_{4,3}\left(x_{1},\ldots,x_{7}\right)\).

The second diagram is a representative of the third category of \(\pi\in S_{\mathrm{nc}}\left(p,q\right)\) in Lemma~\ref{lemma: annular premaps}.  Again, \({\cal U}=\Pi\left(\pi\right)\).  Here,
\[\pi=\left(1,3\right)\left(-3,-1\right)\left(2\right)\left(-2\right)\left(4,-6,-5\right)\left(5,6,-4\right)\left(7\right)\left(-7\right)\]
and \(\pi^{\prime}=\left(1,3\right)\left(2\right)\left(4,6,7\right)\left(5\right)\).  The permutation corresponds to term
\[K_{\left({\cal U},\pi\right)}=K_{2}\left(X_{1},X_{3}\right)K_{1}\left(X_{2}\right)K_{3}\left(X_{4},X_{6}^{T},X_{5}^{T}\right)K_{1}\left(X_{7}\right)\]
with limit
\[\kappa_{2}\left(x_{1},x_{3}\right)\kappa_{1}\left(x_{2}\right)\kappa_{3}\left(x_{4},x_{6}^{t},x_{5}^{t}\right)\kappa_{1}\left(x_{7}\right)=\kappa_{\pi^{\prime}}\left(x_{1},x_{2},x_{3},x_{4},x_{7}^{t},x_{6}^{t},x_{5}^{t}\right)\textrm{.}\]

The third term corresponds to \(\left({\cal U},\pi\right)\) with \({\cal U}=\left\{\left\{1\right\},\left\{2,4\right\},\left\{3,5,6\right\},\left\{7\right\}\right\}\) and \(\pi=\left(1\right)\left(-1\right)\left(2,4\right)\left(-4,-2\right)\left(3\right)\left(-3\right)\left(5,6\right)\left(-6,-5\right)\left(7\right)\left(-7\right)\).

This corresponds to term
\[N^{2}K_{\left({\cal U},\pi\right)}\left(X_{1},\ldots,X_{7}\right)=N^{2}K_{1}\left(X_{1}\right)K_{2}\left(X_{2},X_{4}\right)K_{1,2}\left(X_{3},X_{5},X_{6}\right)K_{1}\left(X_{7}\right)\]
which has \(N\rightarrow\infty\) limit
\[\kappa_{1}\left(x_{1}\right)\kappa_{2}\left(x_{2},x_{4}\right)\kappa_{1,2}\left(x_{3},x_{5},x_{6}\right)\kappa_{1}\left(x_{7}\right)=\kappa_{\left({\cal U},\pi^{\prime}\right)}\left(x_{1},\ldots,x_{7}\right)\textrm{.}\]
\end{example}

\begin{figure}
\label{figure: moment cumulant}
\centering
\begin{tabular}{cc}
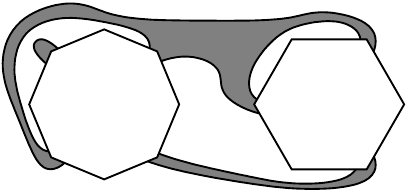\\
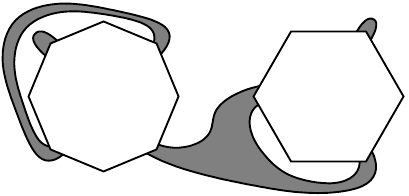\\
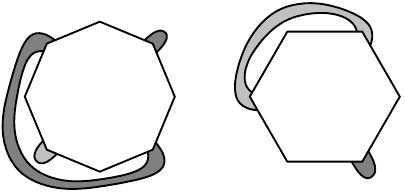
\end{tabular}
\caption{Examples of the three types of diagrams contributing to the highest-order of the cumulant \(k_{2}\left(\mathrm{tr}\left(X_{1}X_{2}X_{3}X_{4}\right),\mathrm{tr}\left(X_{5}X_{6}X_{7}\right)\right)\).}
\end{figure}

\subsection{Higher-order freeness}

\label{subsection: higher-order}

\begin{proposition}
Let \(A\) be an \(m\)th order noncommutative probability space with subalgebras \(A_{1},\ldots,A_{C}\) closed under transposition.  Assume that the cumulants as defined in Proposition~\ref{proposition: higher free cumulants} which are mixed (i.e.\ are applied to elements from more than one of the algebras) vanish.

Let \(m>2\), let \(r_{1},\ldots,r_{m}>0\), and let \(n:=r_{1}+\cdots+r_{m}\).  Let \(x_{1},\ldots,x_{n}\) belong to algebras \(A_{1},\ldots,A_{C}\subseteq A\) such that \(x_{\tau_{r_{1},\ldots,r_{m}}\left(k\right)}\) is always from a different algebra from \(x_{k}\) (i.e.\ the \(x_{k}\) are cyclically alternating under \(\tau_{r_{1},\ldots,r_{m}}\)).  Then the value of \(\varphi_{m}\) applied to cyclically alternating products of centred terms can be calculated from the expansion in Proposition~\ref{proposition: limit cumulant}, excluding any terms in which \({\cal U}\) contains two terms adjacent in \(\tau\) or any singlet:
\[\alpha_{r_{1},\ldots,r_{m}}\left(x_{1},\ldots,x_{n}\right)=\sum_{\substack{\left({\cal U},\pi\right)\in\mathit{PPM}^{\prime}\left(r_{1},\ldots,r_{m}\right)\\\left.\pi\right|_{\left\{k,\tau\left(k\right)\right\}}=\left(k\right)\left(\tau\left(k\right)\right),k\in\left[n\right]\\\left\{k\right\}\neq{\cal U},k\in\left[n\right]}}\kappa_{\left({\cal U},\pi\right)}\left(x_{1},\ldots,x_{n}\right)\textrm{.}\]
\end{proposition}
\begin{proof}
This follows immediately from the hypotheses.
\end{proof}

The characterization of the surviving diagrams is more complicated than in the second-order case.  Some examples of possible diagrams in the third-order case are shown in Figure~\ref{figure: third-order}.  Such diagrams may have \(0\), \(1\), or \(2\) cycles of three elements (in Figure~\ref{figure: third-order}, first two diagrams, third diagram, and fourth diagram, respectively), while all other cycles must have exactly \(2\) elements.  If all cycles have \(2\) elements, and there are two cycles of \(\tau\) not connected directly by cycles of \(\pi\), then it is possible that the outside faces of those cycles of \(\tau\) are not the same cycles of \(\mathrm{Kr}_{\tau}\left(\pi\right)\) (as shown in the second diagram in Figure~\ref{figure: third-order}; see \cite{2020arXiv200309344R} for a more precise defnition of outside faces).

\begin{figure}
\label{figure: third-order}
\centering
\begin{tabular}{cc}
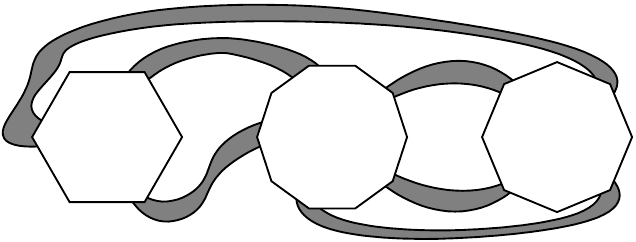\\
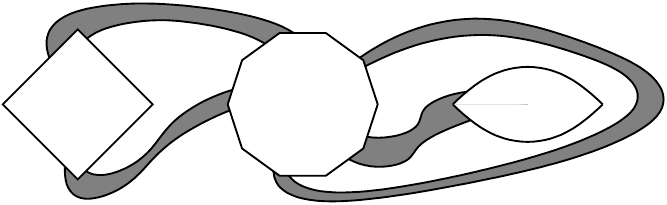\\
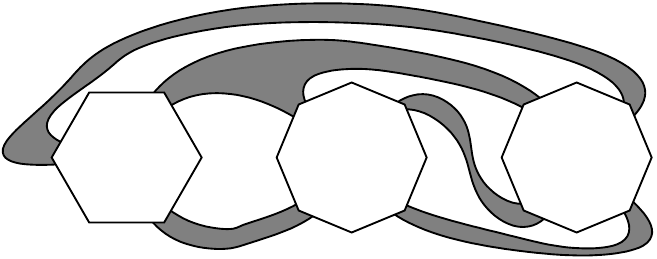\\
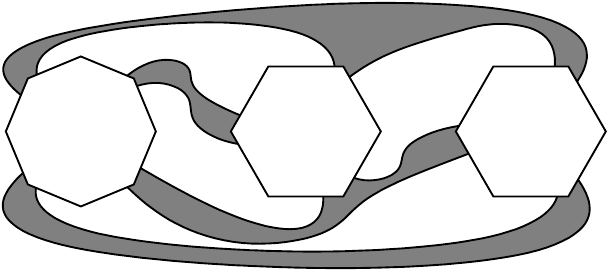
\end{tabular}
\caption{Four examples of diagrams appearing in the third-order expansion.}
\end{figure}

The diagrams in Figure~\ref{figure: third-order} correspond, respectively, to terms:
\[\kappa_{2}\left(x_{1},x_{7}^{t}\right)\kappa_{2}\left(x_{2},x_{9}\right)\kappa_{2}\left(x_{3},x_{6}^{t}\right)\kappa_{2}\left(x_{4}^{t},x_{11}\right)\kappa_{2}\left(x_{5}^{t},x_{12}\right)\kappa_{2}\left(x_{8}^{t}\right)\]
\[\kappa_{2}\left(x_{1},x_{4}\right)\kappa_{2}\left(x_{2},x_{5}\right)\kappa_{2}\left(x_{3},x_{6}\right)\kappa_{2}\left(x_{7},x_{8}\right)\]
\[\kappa_{3}\left(x_{1},x_{10}^{t},x_{6}^{t}\right)\kappa_{2}\left(x_{2},x_{11}^{t}\right)\kappa_{2}\left(x_{3},x_{5}^{t}\right)\kappa_{2}\left(x_{4}^{t},x_{8}^{t}\right)\kappa_{2}\left(x_{7}^{t},x_{9}^{t}\right)\]
\[\kappa_{2}\left(x_{1},x_{6}\right)\kappa_{3}\left(x_{2},x_{10}^{t},x_{5}\right)\kappa_{2}\left(x_{3},x_{8}^{t}\right)\kappa_{3}\left(x_{4},x_{7},x_{8}^{t}\right)\textrm{.}\]

Higher-order diagrams are more complicated to characterize.  In fourth-order diagrams, the permutation need not be connected (consider a pair of spoke diagrams in which the partition connects a spoke from each diagram).  In fifth-order diagrams, it is possible for the permutation to have single-element cycles (consider a pair of spoke diagrams along with a noncrossing diagram on a single cycle whose two single-element cycles are connected by the partition to spokes in the two spoke diagrams).

\section{The quaternionic case}

\label{section: quaternion}

We give here the definition of quaternionic vertex cumulants.  Definitions of quaternionic second-order probability spaces and quaternionic second-order freeness can be found in \cite{2015arXiv150404612R}.  More details on the definition of quaternionic matrix cumulants and their use in computing random matrix quantities can be found in \cite{MR2337139, MR3455672}.  The analogous results for those given for the real case hold for the quaternionic case with appropriate changes to the proofs.

\begin{definition}
The quaternionic Weingarten function \(\mathrm{Wg}^{\mathrm{Sp}\left(N\right)}\) and its cumulants are defined in \cite{MR2217291}; more details and the definition of \(\mathrm{wg}^{\mathrm{Sp}\left(N\right)}\) are given in \cite{MR3455672}.  For large \(N\), the cumulant of the Weingarten function satisfies
\[\mathrm{wg}_{{\cal U},{\cal V}}^{\mathrm{Sp}\left(N\right)}=\gamma^{\mathbb{H}}_{{\cal U},{\cal V}}N^{-2\left(\#\left({\cal U}\right)-\#\left({\cal V}\right)\right)}+O\left(N^{-2\left(\#\left({\cal U}\right)-\#\left({\cal V}\right)\right)-1}\right)\]
where
\[\gamma^{\mathbb{H}}_{{\cal U},{\cal V}}:=\prod_{V\in{\cal V}}\left(-1\right)^{\left|V\right|-\#\left(\left.{\cal U}\right|_{V}\right)}\frac{2\left(2\left|V\right|+\#\left(\left.{\cal U}\right|_{V}\right)-3\right)!}{\left(2\left|V\right|\right)!}\prod_{U\in\left.{\cal U}\right|_{V}}\frac{\left(2\left|U\right|-1\right)!}{\left(\left|U\right|-1\right)!^{2}}\textrm{.}\]
\end{definition}

\begin{definition}
The quaternionic matrix cumulants are given by:
\begin{multline*}
c^{\mathrm{Sp}\left(N\right)}_{\rho}\left(X_{1},\ldots,X_{n}\right)\\:=\sum_{\pi\in\mathrm{PM}\left(n\right)}\left(2N\right)^{\chi_{\rho}\left(\pi\right)-2\#\left(\Pi\left(\rho\right)\right)}\mathrm{wg}^{\mathrm{Sp}\left(N\right)}\left(\Pi\left(\mathrm{Kr}_{\rho}\left(\pi\right)\right)\right)\mathbb{E}\left(\left(\mathrm{Re}\circ\mathrm{tr}\right)_{\pi}\left(X_{1},\ldots,X_{n}\right)\right)\textrm{.}
\end{multline*}
The quaternionic vertex cumulants are given by
\[K^{\mathrm{Sp}\left(N\right)}_{\left({\cal U},\pi\right)}\left(X_{1},\ldots,X_{n}\right)=\sum_{\substack{{\cal V}\in{\cal P}\left(n\right)\\{\cal V}\succeq\Pi\left(\pi\right)}}\mu_{{\cal P}}\left({\cal V},1\right)c^{\mathrm{Sp}\left(N\right)}_{\left({\cal V},\pi\right)}\left(X_{1},\ldots,X_{n}\right)\textrm{.}\]
Quaternionic higher-order free cumulants are given by
\begin{multline*}
\kappa^{\mathbb{H}}_{r_{1},\ldots,r_{m}}\left(x_{1},\ldots,x_{n}\right)\\:=\sum_{\substack{\left({\cal U},\pi\right)\in\mathit{PPM}^{\prime}\left(\tau\right)\\{\cal V}\succeq\Pi\left(\mathrm{Kr}_{\tau}\left(\pi\right)\right),{\cal V}\in\Gamma\left(\Pi\left(\tau\right)\vee{\cal U},\Pi\left(\mathrm{Kr}_{\tau}\left(\pi\right)\right)\right)\\\Pi\left(\tau\right)\vee{\cal U}\vee{\cal V}=1}}\gamma^{\mathbb{H}}_{\Pi\left(\mathrm{Kr}_{\tau}\left(\pi\right)\right),{\cal V}}\alpha_{\left({\cal U},\pi\right)}\left(x_{1},\ldots,x_{n}\right)\textrm{.}
\end{multline*}
\end{definition}

\bibliography{paper}
\bibliographystyle{plain}

\end{document}